\documentclass[a4paper,11pt]{amsart}

\usepackage{a4wide,amsmath,amsfonts,amssymb,amsthm}
\usepackage{bbm}
\usepackage[hyperfootnotes=false]{hyperref}

\allowdisplaybreaks[2]

\numberwithin{equation}{section}

\newtheorem{theorem}{Theorem}[section]
\newtheorem{lemma}[theorem]{Lemma}

\newtheorem{remark}[theorem]{Remark}

\newtheorem{assumption}[theorem]{Assumption}

\newcommand{\dd}{\,\mathrm{d}}
\newcommand{\R}{\mathbb{R}}
\newcommand{\N}{\mathbb{N}}
\newcommand{\E}{\mathbb{E}}

\title[Mean-field stochastic Volterra equations]{Mean-field stochastic Volterra equations}

\author[Pr{\"o}mel]{David J. Pr{\"o}mel}
\address{David J. Pr{\"o}mel, University of Mannheim, Germany}
\email{proemel@uni-mannheim.de}

\author[Scheffels]{David Scheffels}
\address{David Scheffels, University of Mannheim, Germany}
\email{dscheffe@mail.uni-mannheim.de}

\date{\today}

\begin{document}

\begin{abstract}
  The well-posedness is established for multi-dimensional mean-field stochastic Volterra equations with Lipschitz continuous coefficients and allowing for singular kernels as well as for one-dimensional mean-field stochastic Volterra equations with H{\"o}lder continuous diffusion coefficients and sufficiently regular kernels. In these different settings, quantitative, pointwise propagation of chaos results are derived for the associated Volterra type interacting particle systems.
\end{abstract}

\maketitle





\section{Introduction}

Mean-field stochastic differential equations (mean-field SDEs), also known as McKean--Vlasov stochastic differential equations, provide mathematical descriptions of random systems of interacting particles, whose time evolutions depend, in some manner, on the probability distribution of the entire systems. A crucial reason for the frequent use of mean-field SDEs in applied mathematics is the fact that they allow for modelling the phenomena of ``propagation of chaos'' of large interacting particle systems. Recall, on a microscopic scale the trajectory of each individual particle can often be appropriately modelled by a stochastic process. However, when the number of particles becomes very large, the microscopic scale usually contains too much information, making the interaction of individual particles intractable. Fortunately, sending the number of particles to infinity, propagation of chaos states that the behavior of an individual particle depends only on the probability distribution of the entire system, i.e., on the macroscopic scale the interaction of individual particles becomes negligible.

Mean-field SDEs as well as propagation of chaos originated in statistical physics and were first studied by Kac~\cite{Kac56}, McKean~\cite{McKean66} and Vlasov~\cite{Vlasov1968}. By now, these concepts have found a wide range of applications in a variety of fields such as physics, finance and data science. We refer, e.g., to \cite{Sznitman1991,Jabin2017,Carmona2018,Carmona2018b,Chaintron2022,Chaintron2022b} for comprehensive introductions to mean-field SDEs and their numerous applications. Except a very small number of publications, like the rough path based approaches to mean-field SDEs \cite{Bailleul2020,Coghi2020,Bailleul2021}, the vast majority of literature on mean-field SDEs and propagation of chaos is restricted to Markovian systems of interacting particles, i.e. the behavior of each particle has to be independent of all past states of the systems. On the contrary, it is well observed that many real-world dynamical systems do have memory effects and, thus, do indeed depend on past states of the underlying systems. Well-known examples of such systems are the growth of populations, the spread of epidemics and turbulence flows.

Classical mathematical models for random dynamical systems with memory effects are given by stochastic Volterra equations (SVEs), as introduced in the seminal works of Berger and Mizel~\cite{Berger1980a,Berger1980b}, see also e.g. \cite{Protter1985,Pardoux1990}. While SVEs allow for generating non-Markovian stochastic processes, the solutions of SVEs, in contrast to mean-field SDEs, do not depend directly on the probability distributions of the generated random systems.

In the present paper we aim to unify the theories of mean-field stochastic differential equations and stochastic Volterra equations, which enables to combine the desirable modelling advantages of both classes of equations. More precisely, we introduce mean-field stochastic Volterra equations (mean-field SVEs)
\begin{equation}\label{eq:Intro MVSVE}
  X_t = X_0+\int_0^t K_{\mu}(s,t)\mu(s,X_s,\mathcal{L}(X_s))\dd s+\int_0^t K_{\sigma}(s,t)\sigma(s,X_s,\mathcal{L}(X_s))\dd B_s,\quad t\in [0,T],
\end{equation}
where $X_0$ is a random variable, $B$ is a Brownian motion, and the coefficients $\mu, \sigma$ as well as the kernels $K_\mu, K_\sigma$ are measurable functions. Here, $\mathcal{L}(X_s)$ denotes the law of the random variable~$X_s$. In words, mean-field SVEs are a class of stochastic integral equations that describe the dynamics of random systems with both nonlinear interactions and memory effects. They constitute a generalization of mean-field SDEs and of classical SVEs. Notice that a solution to the mean-field SVE~\eqref{eq:Intro MVSVE} is, in general, neither a Markov process nor a semimartingale.

Our first contribution is to establish the (strong) well-posedness of the mean-field SVE~\eqref{eq:Intro MVSVE}, meaning that there exists a unique strong solution to~\eqref{eq:Intro MVSVE}, under two sets of assumptions. On the one hand, we show the existence of a unique solution to the mean-field SVE~\eqref{eq:Intro MVSVE} in a multi-dimensional setting with standard assumptions on the kernels and coefficients, i.e. we assume some integrability on the kernels as well as Lipschitz continuity and a linear growth condition on the coefficients, cf. e.g. \cite{Wang2008,Carmona2016}. The proof is based on a classical fixed point argument in combination with techniques from the theories of mean-field SDEs and SVEs. On the other hand, we show the existence of a unique solution to the mean-field SVE~\eqref{eq:Intro MVSVE} in a one-dimensional setting, assuming sufficiently smooth kernels and H{\"o}lder continuous diffusion coefficients which are independent of the law of the solution. To that end, we rely on a Yamada--Watanabe approach~\cite{Yamada1971}, as recently generalized in \cite{AbiJaber2019, Promel2023} to SVEs with sufficiently smooth kernels. As comparison, for well-posedness results in case of mean-field SDEs we refer to \cite{Bahlali2020,Kalinin2022,Huang2023} and in case of SVEs to \cite{Wang2008,AbiJaber2019,Promel2023}. Furthermore, let us remark that a specific type of mean-field SVEs was studied in \cite{Shi2013}, where the coefficients may depend on the law of the solution but only through an expectation operator.

Our second contribution is to establish quantitative, pointwise propagation of chaos results of Volterra-type  systems of interacting particles. In words, sending the number of Volterra-type interacting particles to infinity, we obtain a macroscopic description of the systems based on a mean-field stochastic Volterra equation. The developed approach is based on a synchronous coupling method, as it was initiated by McKean~\cite{McKean1967} and extended by Sznitman~\cite{Sznitman1991}. In the case of mean-field SDEs, synchronous coupling methods are widely used for systems that are described by systems of McKean--Vlasov diffusions, and often lead to pathwise propagation of chaos, see e.g. \cite[Theorem~3.20]{Chaintron2022}, \cite[Theorem~1.10]{Carmona2016} and \cite{Huang2023}. In the present case of mean-field SVEs, implementing a synchronous coupling method becomes more challenging as the underlying McKean--Vlasov processes are of Volterra type and, thus, in general, lack the semimartingale and Markov property. As for the presented well-posedness theory of mean-field SVEs, we distinguish between the aforementioned multi- and one-dimensional setting. The pointwise natures of the presented propagation of chaos results for mean-field SVEs is caused by the non-availability of a Burkholder--Davis--Gundy inequality in the multi-dimensional setting and by the H{\"o}lder continuity of the diffusion coefficients in the one-dimensional setting. The latter setting requires to combine the synchronous coupling method with a Yamada--Watanabe approach.

\medskip

\noindent \textbf{Organization of the paper:} In Section~\ref{sec:main result} we present the main results regarding the well-posedness and propagation of chaos for mean-field stochastic Volterra equations. Section~\ref{sec:well-posedness of SVEs} provides some necessary well-posedness results for ordinary stochastic Volterra equations. The proofs of the main results are contained in Section~\ref{sec:1}, \ref{sec:2} and \ref{sec:3}.

\medskip

\noindent\textbf{Acknowledgments:} D. Scheffels gratefully acknowledges financial support by the Research Training Group ``Statistical Modeling of Complex Systems'' (RTG 1953) funded by the German Science Foundation (DFG). D. J. Pr{\"o}mel and D. Scheffels would like to thank P. Nikolaev for fruitful discussions helping to improve the present work.

\section{Main results: well-posedness and propagation of chaos}\label{sec:main result}

Let $T\in (0,\infty)$, $d,m\in\N$, and let $(\Omega,\mathcal{F},(\mathcal{F}_t)_{t\in [0,T]},\mathbb{P})$ be a filtered probability space, which satisfies the usual conditions. Suppose $B=(B_t)_{t\in [0,T]}$ is an $m$-dimensional Brownian motion with respect to $(\mathcal{F}_t)_{t\in [0,T]}$. The law of a random variable~$X$ is denoted by $\mathcal{L}(X)$ and, for $p\geq 1$, the space of probability measures on $\R^d$ with finite $p$-th moments by $\mathcal{P}_{p}(\R^d)$. Let $C([0,T];\R^d)$ be the space of continuous functions from $[0,T]$ to $\R^d$, equipped with the supremum norm $\|\cdot\|_\infty$, and $\mathcal{P}_p(C([0,T];\R^d))$ be the space of probability measures on $C([0,T];\R^d)$ with finite $p$-th moment. For $\rho,\tilde{\rho}\in\mathcal{P}_p(\R^d)$, we write $W_p(\rho,\tilde{\rho})$ for the $p$-Wasserstein distance between $\rho$ and $\tilde{\rho}$, see \cite[Chapter~5]{Carmona2018} for its definition, and with a slight abuse of notation, for $\rho,\tilde{\rho} \in \mathcal{P}_p(C([0,T];\R^d))$ we define the $p$-Wasserstein distance by
\begin{equation*}
  W_p(\rho,\tilde{\rho}) = \inf_{\pi \in \Pi(\rho,\tilde{\rho})} \Big[ \int_{C([0,T];\R^d)^2} \|x-y\|_\infty^p \dd\pi(x,y)\Big]^{1/p},
\end{equation*}
where $\Pi(\rho,\tilde{\rho})$ denotes the set of all probability measures on $C([0,T];\R^d)^2$ with marginal distributions given by $\rho$ and $\tilde{\rho}$, respectively. The space $\R^d$ is always equipped with the Euclidean norm~$| \cdot |$ and on the space $\R^{d\times m} $ we use the Frobenius norm also denoted by~$| \cdot |$. Moreover, we set $\Delta_T:=\lbrace (s,t)\in [0,T]\times [0,T]\colon \, 0\leq s\leq t\leq T \rbrace$ and use the notation $A_{\eta}\lesssim B_{\eta}$ for a generic parameter~$\eta$, meaning that $A_{\eta}\le CB_{\eta}$ for some constant $C>0$ independent of~$\eta$.

\medskip

We consider the $d$-dimensional mean-field stochastic Volterra equation
\begin{equation}\label{eq:MVSVE}
  X_t = X_0+\int_0^t K_{\mu}(s,t)\mu(s,X_s,\mathcal{L}(X_s))\dd s+\int_0^t K_{\sigma}(s,t)\sigma(s,X_s,\mathcal{L}(X_s))\dd B_s,\quad t\in [0,T],
\end{equation}
where $X_0$ is a $d$-dimensional, $\mathcal{F}_0$-measurable random variable, which is independent of $B$, the coefficients $\mu\colon[0,T]\times\R^d\times\mathcal{P}_{p}(\R^d)\to\R^d$, $\sigma\colon[0,T]\times\R^d\times\mathcal{P}_{p}(\R^{d})\to\R^{d\times m}$ and the kernels $K_\mu, K_\sigma\colon \Delta_T\to \R$ are measurable functions. The integral $\int_0^t K_{\sigma}(s,t)\sigma(s,X_s,\mathcal{L}(X_s))\dd B_s$ is defined as a stochastic It{\^o} integral.

\medskip

Let us briefly recall the concepts of well-posedness, strong solutions and pathwise uniqueness. We use, for measure spaces $\mathcal{X},\mathcal{Y}$ and $p\geq 1$, the notation $L^p=L^p(\mathcal{X};\mathcal{Y})$ for the space of all $\mathcal{Y}$-valued, measurable, $p$-integrable functions on $\mathcal{X}$ and, for two Banach spaces $\mathcal{X},\mathcal{Y}$, $C(\mathcal{X};\mathcal{Y})$ for the space of all $\mathcal{Y}$-valued, continuous functions on $\mathcal{X}$. An $(\mathcal{F}_t)_{t\in[0,T]}$-progressively measurable stochastic process $(X_t)_{t\in [0,T]}$ in $L^p(\Omega\times [0,T];\R^d)$, on the given probability space $(\Omega,\mathcal{F},(\mathcal{F}_t)_{t\in[0,T]},\mathbb{P})$, is called a \textit{(strong) $L^p$-solution} of the mean-field SVE~\eqref{eq:MVSVE} if
\begin{equation*}
  \int_0^t (|K_\mu(s,t)\mu(s,X_s,\mathcal{L}(X_s))|+|K_\sigma(s,t)\sigma(s,X_s,\mathcal{L}(X_s))|^2 )\dd s<\infty \quad \text{for all }t\in[0,T],
\end{equation*}
and the integral equation~\eqref{eq:MVSVE} holds $\mathbb{P}$-almost surely. We say \textit{pathwise uniqueness in} $L^p$ holds for the mean-field SVE~\eqref{eq:MVSVE} if $\mathbb{P}(X_t=\tilde{X}_t, \,\forall t\in [0,T])=1$ for any two $L^p$-solutions $(X_t)_{t\in[0,T]}$ and $(\tilde{X}_t)_{t\in[0,T]}$ of \eqref{eq:MVSVE} defined on the same probability space $(\Omega,\mathcal{F},(\mathcal{F}_t)_{t\in[0,T]},\mathbb{P})$. We say that the mean-field SVE~\eqref{eq:MVSVE} is \textit{well-posed in $L^p$} (or that there exists a \textit{unique $L^p$-solution}) for $p\geq 1$ if there exists a strong $L^p$-solution to \eqref{eq:MVSVE} and pathwise uniqueness in $L^p$ holds.

\medskip

In the following we distinguish between a multi-dimensional and a one-dimensional setting since these settings allow to establish well-posedness of the mean-field SVE~\eqref{eq:MVSVE} with different regularity assumptions on the kernels and coefficients. The main existence and uniqueness results regarding mean-field SVEs as well as propagation of chaos are stated in Subsection~\ref{subs:Lipschitz} and~\ref{subs:Holder}. In the multi-dimensional setting (Subsection~\ref{subs:Lipschitz}) we make standard Lipschitz assumptions on the coefficients $\mu,\sigma$, whereas in the one-dimensional setting (Subsection~\ref{subs:Holder}) we assume that $\mu$ is Lipschitz continuous but allow $\sigma$ to be only H{\"o}lder continuous. We prove the corresponding results in Section~\ref{sec:1}, \ref{sec:2} and \ref{sec:3}.

\subsection{Mean-field SVEs with Lipschitz continuous coefficients}\label{subs:Lipschitz}

In this subsection we consider the multi-dimensional stochastic Volterra equation~\eqref{eq:MVSVE} with dimensions $d,m\in\N$ and coefficients $\mu,\sigma$ that are Lipschitz continuous in the space and distributional component, uniformly in the time component, allowing for potentially singular kernels. We start by stating the assumptions on the kernels.

\begin{assumption}\label{ass:kernels1}
  Assume there are constants $\gamma\in (0,\frac{1}{2}]$, $\epsilon>0$ and $L>0$, such that $K_\mu, K_\sigma\colon \Delta_T\to \R$ are measurable functions fulfilling
  \begin{align*}
    &\int_0^t |K_{\mu}(s,t')-K_{\mu}(s,t)|^{1+\epsilon}\dd s + \int_t^{t'} |K_{\mu}(s,t')|^{1+\epsilon}\dd s \leq L|t'-t|^{\gamma(1+\epsilon)},\\
    &\int_0^t |K_{\sigma}(s,t')-K_{\sigma}(s,t)|^{2+\epsilon}\dd s  + \int_t^{t'} |K_{\sigma}(s,t')|^{2+\epsilon}\dd s  \leq L|t'-t|^{\gamma(2+\epsilon)},
  \end{align*}
  for all $(t,t^\prime)\in \Delta_T$.
\end{assumption}

Note that Assumption~\ref{ass:kernels1} allows for singular kernels, like the fractional convolutional kernel $K(s,t)=(t-s)^{-\alpha}$ for $\alpha\in (0,1/2)$ and the examples provided in \cite[Example~1.3]{AbiJaber2021}. Moreover, let for $\epsilon>0$ given by Assumption~\ref{ass:kernels1}, the fixed parameter $\delta>2$ be defined by
\begin{equation}\label{eq:delta}
  \delta:=\frac{4+2\epsilon}{\epsilon},
\end{equation}
such that
\begin{equation}\label{eq:delta_Holder}
  \frac{2}{2+\epsilon}+\frac{2}{\delta}=1.
\end{equation}
In the following we use the $\delta$-Wasserstein distance on the space $\mathcal{P}_\delta(\R^d)$ of probability measures on $\R^d$ with finite $\delta$-th moments. Relying on the $\delta$-Wasserstein distance, we specify the assumptions on the regularity of the coefficients $\mu$ and $\sigma$, which are a classical linear growth condition and a Lipschitz assumption.

\begin{assumption}\label{ass:coefficients1}
  Let $\mu\colon [0,T]\times\R^d\times\mathcal{P}_{\delta}(\R^d)\to\R^d $ and $\sigma\colon [0,T]\times\R^d\times\mathcal{P}_{\delta}(\R^d)\to\R^{d\times m} $ be measurable functions such that:
  \begin{enumerate}
    \item[(i)] for any bounded set $\mathcal{K}\subset \mathcal{P}_{\delta}(\R^d)$, there is a constant $C_{\mathcal{K}}>0$, such that the linear growth condition
    \begin{equation*}
      |\mu(t,x,\rho)|+|\sigma(t,x,\rho)|\leq C_{\mathcal{K}} (1+|x|)
    \end{equation*}
    holds for all $\rho \in \mathcal{K}$, $t\in[0,T]$ and $x\in\R^d$;
    \item[(ii)] $\mu$ and $\sigma$ are Lipschitz continuous in $x$ and in $\rho$ w.r.t. the $\delta$-Wasserstein distance, uniformly in $t$, i.e. there is a constant $C_{\mu,\sigma}>0$ such that
    \begin{equation*}
      |\mu(t,x,\rho)-\mu(t,\tilde{x},\tilde{\rho})|+|\sigma(t,x,\rho)-\sigma(t,\tilde{x},\tilde{\rho})|
      \leq C_{\mu,\sigma}\big(|x-\tilde{x}|+W_{\delta}(\rho,\tilde{\rho})\big),
    \end{equation*}
    holds for all $t\in [0,T]$, $x,\tilde{x}\in \R^d$, and $\rho,\tilde{\rho}\in \mathcal{P}_{\delta}(\R^d)$.
  \end{enumerate}
\end{assumption}

Our first result is the well-posedness of the mean-field stochastic Volterra equation~\eqref{eq:MVSVE}.

\begin{theorem}\label{thm:well-posedness}
  Suppose that the initial value $X_0$ is in $L^p(\Omega;\R^d)$, the kernels $K_\mu, K_\sigma$ fulfill Assumption~\ref{ass:kernels1}, the coefficients $\mu,\sigma$ fulfill Assumption~\ref{ass:coefficients1}, and $p>\max\lbrace \frac{1}{\gamma},1+\frac{2}{\epsilon} \rbrace$, where $\gamma\in(0,\frac{1}{2}]$ and $\epsilon>0$ are given by Assumption~\ref{ass:kernels1}. Then, the mean-field stochastic Volterra equation~\eqref{eq:MVSVE} is well-posed in $L^p$. Moreover, for any $q\geq p$, if $X_0\in L^q(\Omega;\R^d)$, the unique $L^p$-solution $X$ of~\eqref{eq:MVSVE} satisfies
  \begin{equation}\label{fin_moments}
    \sup\limits_{t\in[0,T]}\E[|X_t|^{q}]<\infty.
  \end{equation}
\end{theorem}

Our second result is propagation of chaos for mean-field stochastic Volterra equations, i.e. we show that the unique $L^p$-solution to the mean-field stochastic Volterra equation~\eqref{eq:MVSVE} is the limit $N\to\infty$ of the solutions to the following system of $N$ mean-field stochastic Volterra equations
\begin{equation}\label{eq:X^n}
  X_t^{N,i}=X_0^i + \int_0^t K_\mu(s,t)\mu(s,X_s^{N,i},\bar{\rho}_s^N)\dd s + \int_0^t K_\sigma(s,t) \sigma(s,X_s^{N,i},\bar{\rho}_s^N)\dd B_s^i, \quad t\in[0,T],
\end{equation}
for $i\in\lbrace 1,\dots,N\rbrace$, where $\bar{\rho}_t^N:=\frac{1}{N}\sum\limits_{i=1}^N \delta_{X_t^{N,i}}$ is the empirical distribution of $(X_t^{N,i})_{i=1,\dots,N}$, $(X_0^i)_{i\in\N}\subset L^q(\Omega;\R^d)$ is a sequence of $\mathcal{F}_0$-measurable, independent and identically distributed random variables for some $q>4$, and $(B^i)_{i\in\N}$ is a sequence of independent $m$-dimensional Brownian motions, which are all defined on the given probability space $(\Omega,\mathcal{F},(\mathcal{F}_t)_{t\in [0,T]},\mathbb{P})$. Strong $L^p$-solutions, pathwise uniqueness in $L^p$ and well-posedness in $L^p$ for the system~\eqref{eq:X^n} of mean-field SVEs is defined analogously to \eqref{eq:MVSVE} and $\delta_x$ denotes the Dirac measure at $x$ for $x\in \R^d$. Moreover, for $i\in\N$, let $\underbar{X}^i$ be the solution of the mean-field SVE~\eqref{eq:MVSVE} with the initial condition~$X_0^i$ and driving Brownian motion~$B^i$. In the present multi-dimensional setting, we obtain the following convergence result.

\begin{theorem}[Volterra propagation of chaos]\label{thm:propagation}
  Suppose Assumption~\ref{ass:kernels1} and \ref{ass:coefficients1}, and that the sequence of initial conditions $(X_0^i)_{i\in\N}\subset L^q(\Omega;\R^d)$ for some $q>\max\lbrace p,2\delta \rbrace$ and $p>\max\lbrace \frac{1}{\gamma},1+\frac{2}{\epsilon} \rbrace$, where $\delta$ is defined in \eqref{eq:delta}. Then, the system~\eqref{eq:X^n} of mean-field SVEs is well-posed in $L^p$ for every $N\geq 1$, where the unique $L^p$-solution is denoted by $(X_t^{N,i})_{i=1,\dots,N}$. Moreover, it holds
  \begin{equation}\label{eq:Thm2}
    \lim\limits_{N\to\infty}\bigg( \max\limits_{1\leq i\leq N}\Big( \sup\limits_{ t \in [0,T]}\E[|X_t^{N,i}-\underbar{X}_t^i|^\delta] \Big) +\sup\limits_{ t \in [0,T]}\E\Big[ W_\delta\Big( \frac{1}{N}\sum\limits_{i=1}^N\delta_{X_t^{N,i}},\mathcal{L}(\underbar{X}_t^1) \Big)^\delta \Big] \bigg)=0.
  \end{equation}
\end{theorem}

The rate of convergence in \eqref{eq:Thm2} is explicitly stated in the next lemma.

\begin{lemma}\label{lem:rates}
  Supposing the assumptions and notations of Theorem~\ref{thm:propagation}, it holds that
  \begin{equation}\label{eq:rates}
    \max\limits_{1\leq i\leq N}\Big( \sup\limits_{t \in [0,T]}\E[|X_t^{N,i}-\underbar{X}_t^i|^\delta] \Big) +\sup\limits_{t \in [0,T]}\E\Big[ W_\delta\Big( \frac{1}{N}\sum\limits_{i=1}^N\delta_{X_t^{N,i}},\mathcal{L}(\underbar{X}_t^1) \Big)^\delta \Big] \lesssim \varepsilon_N,
  \end{equation}
  where $(\varepsilon_N)_{N\in\N}$ is given by
  \begin{align}\label{def:varepsilon}
    \varepsilon_N = \Bigg\{\begin{array}{lr}
        N^{-1/2}, & \text{if } d<2\delta\\
        N^{-1/2}\log_2(1+N), & \text{if } d=2\delta\\
        N^{-\delta/d}, & \text{if } d>2\delta
    \end{array}.
  \end{align}
\end{lemma}

\begin{remark}\label{rem:rate of convergence 1}
  The rates of convergence obtained in \eqref{def:varepsilon} are analogue to the classical rates for ordinary mean-field SDEs with Lipschitz coefficients (see \cite[Theorem~3.20]{Chaintron2022}), using $W_\delta(\cdots)^\delta$ instead of $W_2(\cdots)^2$ and, consequently, replacing the exponent $2/d$ by $\delta/d$ in \eqref{def:varepsilon}. Note that in the case of ordinary mean-field SDEs one obtains a pathwise propagation of chaos result (meaning that the $\sup$ in \eqref{eq:rates} is inside the expectation operators), which is a stronger type of convergence than the pointwise convergence presented in Theorem~\ref{thm:propagation}. This weaker type of convergence is caused by the missing availability of the standard Burkholder--Davis--Gundy inequality for the solutions of stochastic Volterra equations since they are, in general, not semimartingales. However, the rates of convergence provided in Lemma~\ref{lem:rates} seem to be optimal for synchronous coupling methods, since it is shown in \cite[Theorem~1 and there after]{Fournier2015} that for terms of the form $\E[W_\delta(\bar{\rho}_N,\rho)^\delta]$ the rates in \eqref{def:varepsilon} are sharp. Consequently, optimality could be only lost in the inequalities \eqref{ineq:convRHS} or \eqref{ineq:settingII2}, which, at least in general, appears not to be the case.
\end{remark}

\subsection{Mean-field SVEs with H{\"o}lder continuous diffusion coefficients}\label{subs:Holder}

In this subsection we consider mean-field SVEs in a one-dimensional setting, i.e. we assume $d=m=1$. This allows to relax the Lipschitz assumption on the diffusion coefficient~$\sigma$ to H{\"o}lder continuity in the space variable, provided that~$\sigma$ is independent of the distribution of the solution and that the kernels are sufficiently regular. More precisely, we consider the one-dimensional mean-field stochastic Volterra equation
\begin{equation}\label{eq:MVSVE2}
  X_t = X_0+\int_0^t K_{\mu}(s,t)\mu(s,X_s,\mathcal{L}(X_s))\dd s+\int_0^t K_{\sigma}(s,t)\sigma(s,X_s)\dd B_s,\quad t\in [0,T],
\end{equation}
where $(B_t)_{t\in[0,T]}$ is a one-dimensional Brownian motion, $X_0$ is an $\mathcal{F}_0$-measurable random variable, the coefficients $\mu\colon[0,T]\times\R\times\mathcal{P}_{p}(\R)\to\R$, $\sigma\colon[0,T]\times\R\to\R$ and the kernels $K_\mu, K_\sigma\colon \Delta_T\to \R$ are measurable functions. We consider two different sets of assumptions on the kernels and on the initial condition.

\begin{assumption}\label{ass:kernels2}
  Let $\gamma\in (0,\frac{1}{2}]$ and $\epsilon >0$. Let $X_0$ be an $\mathcal{F}_0$-measurable random variable and $K_\mu, K_\sigma\colon \Delta_T\to \R$ be continuous functions such that:
  \begin{enumerate}
    \item[(i)] $K_\mu(s,\cdot)$ is absolutely continuous for every $s\in [0,T]$ and $\partial_2 K_\mu $ is bounded on~$\Delta_T$;

    \item[(ii)] $K_\sigma(\cdot,t)$ is absolutely continuous for every $t\in [0,T]$, $K_\sigma(s,\cdot)$ is absolutely continuous for every $s\in [0,T]$ with $\partial_2 K_\sigma\in L^{2}(\Delta_T)$, and $\partial_2 K_\sigma(\cdot,t)$ is absolutely continuous for every $t\in[0,T]$. Furthermore, there is a constant $C_1>0$ such that $|K_\sigma(t,t)|\geq C_1$ for any $t\in [0,T]$, and there exists $C_2>0$ such that
    \begin{align*}
      &\int_0^s |K_{\sigma}(u,t)-K_{\sigma}(u,s)|^{2+\epsilon}\dd u \leq C_2|t-s|^{\gamma (2+\epsilon)}
      \text{ and}\\
      &|\partial_1 K_\sigma (s,t)| +|\partial_2 K_\sigma (s,s)| +\int_s^t |\partial_{21}K_\sigma(s,u)|\dd u\leq C_2
    \end{align*}
    hold for any $(s,t)\in\Delta_T$;

    \item[(iii)] $X_0 \in L^p(\Omega;\R)$ for $p>\max\lbrace \frac{1}{\gamma},1+\frac{2}{\epsilon} \rbrace$.
  \end{enumerate}
\end{assumption}

Instead of Assumption~\ref{ass:kernels2}, we can alternatively require $K_\mu,K_\sigma$ and $X_0$ to fulfill the following assumption, where the kernels are supposed to be convolutional.

\begin{assumption}\label{ass:kernels3}
  Let $X_0$ be an $\mathcal{F}_0$-measurable random variable and $K_\mu, K_\sigma\colon \Delta_T\to \R$ be continuous functions such that:
  \begin{enumerate}
    \item[(i)] $K_\mu(s,t)=K_\sigma(s,t)=\tilde{K}(t-s)$ for some $\tilde{K}\in C^1([0,T];\R)$;

    \item[(ii)] $X_0 \in L^p(\Omega;\R)$ for $p>2$.
  \end{enumerate}
\end{assumption}

Next, we formulate the assumptions on the coefficients.
  
\begin{assumption}\label{ass:coefficients2}
  Let $\mu\colon [0,T]\times\R\times\mathcal{P}_{1}(\R)\to\R $ and $\sigma\colon [0,T]\times\R\to\R$ be measurable functions such that:
  \begin{enumerate}
	\item[(i)] for any bounded set $\mathcal{K}\subset\mathcal{P}_{1}(\R)$, there is a constant $C_{\mathcal{K}}>0$, such that the linear growth condition
	\begin{equation*}
	  |\mu(t,x,\rho)|+|\sigma(t,x)|\leq C_\mathcal{K} \rho(1+|x|)
    \end{equation*}
	holds for all $\rho \in \mathcal{K}$, $t\in[0,T]$ and $x\in\R$;
    \item[(ii)] $\mu$ is Lipschitz continuous in $x$ and $\rho$ w.r.t. the $1$-Wasserstein distance, uniformly in $t$, i.e. there is a constant $C_{\mu}>0$ such that
    \begin{equation*}
      |\mu(t,x,\rho)-\mu(t,\tilde{x},\tilde{\rho})|\leq C_{\mu}\big(|x-\tilde{x}|+W_{1}(\rho,\tilde{\rho})\big),
    \end{equation*}
    holds for all $t\in [0,T]$, $x,\tilde{x}\in \R$ and $\rho,\tilde{\rho}\in \mathcal{P}_{1}(\R)$, and $\sigma$ is H{\"o}lder continuous of order $\frac{1}{2}+\xi$ for some $\xi\in [0,\frac{1}{2}]$ in $x$ uniformly in $t$, i.e. there is a constant $C_{\sigma}>0$ such that
    \begin{equation*}
      |\sigma(t,x)-\sigma(t,\tilde{x})|\leq C_{\sigma}|x-\tilde{x}|^{\frac{1}{2}+\xi},
    \end{equation*}
    holds for all $t\in [0,T]$ and $x,\tilde{x}\in \R$.
  \end{enumerate}
\end{assumption}

First, we establish the well-posedness of the mean-field stochastic Volterra equation~\eqref{eq:MVSVE2} with H{\"o}lder continuous diffusion coefficients. Its proof is based on a Yamada--Watanabe type approach~\cite{Yamada1971}, which requires essentially a one-dimensional setting and leads to the stronger assumptions on the kernels. Moreover, note that the H{\"o}lder continuous diffusion coefficients are required to be independent of the law of the solution, which is essentially a standard assumption for ordinary mean-field stochastic differential equations as it appears to be a necessary assumption to implement a Yamada--Watanabe type approach, cf. \cite{Kalinin2022}.

\begin{theorem}\label{thm:well-posedness2}
  Suppose Assumption~\ref{ass:coefficients2}, and the kernels $K_{\mu},K_{\sigma}$ and the initial condition~$X_0$ satisfy Assumption~\ref{ass:kernels2} or Assumption~\ref{ass:kernels3} with $p$ given as therein. Then, the mean-field stochastic Volterra equation~\eqref{eq:MVSVE2} is well-posed in $L^p$. Moreover, for any $q\geq p$, if $X_0\in L^q(\Omega;\R^d)$, the unique solution $X$ of~\eqref{eq:MVSVE2} satisfies
  \begin{equation*}
    \sup\limits_{t\in[0,T]}\E[|X_t|^{q}]<\infty.
  \end{equation*}
\end{theorem}

Secondly, we establish propagation of chaos for one-dimensional stochastic mean-field SVEs with H{\"o}lder continuous diffusion coefficients. To that end, we consider the symmetric system of $N$ mean-field stochastic Volterra equations
\begin{equation}\label{eq:X^n2}
  X_t^{N,i}=X_0^i + \int_0^t K_\mu(s,t)\mu(s,X_s^{N,i},\bar{\rho}_s^N)\dd s + \int_0^t K_\sigma(s,t) \sigma(s,X_s^{N,i})\dd B_s^i, \quad t\in[0,T],
\end{equation}
for $i\in\lbrace 1,\dots,N\rbrace$, where $(X_0^i)_{i\in\N}\subset L^p(\Omega;\R)$ is an i.i.d. sequence of initial conditions, and $(B^i)_{i\in\N}$ is a sequence of independent one-dimensional Brownian motions. Moreover, for $i\in\N$, $\underbar{X}^i$ denotes the solution of the mean-field SVE~\eqref{eq:MVSVE2} with initial condition~$X_0^i$ and driving Brownian motion~$B^i$. In the present one-dimensional setting, we obtain the following convergence result.

\begin{theorem}[Volterra propagation of chaos]\label{thm:propagation2}
  Suppose Assumption~\ref{ass:coefficients2}, and the kernels $K_{\mu},K_{\sigma}$ and the initial conditions~$X_0^i$, for $i\in\N$, satisfy Assumption~\ref{ass:kernels2} or Assumption~\ref{ass:kernels3} with $p$ given as therein. Then, the system~\eqref{eq:X^n2} of mean-field SVEs is well-posed in $L^p$, where the unique $L^p$-solution is denoted by $(X_t^{N,i})_{i=1,\dots,N}$ for every $N\geq 1$. Moreover, it holds
  \begin{equation}\label{eq:Thm2ii}
    \lim\limits_{N\to\infty}\bigg( \max\limits_{1\leq i\leq N}\Big( \sup\limits_{ t\in [0,T]}\E[|X_t^{N,i}-\underbar{X}_t^i|] \Big) +\sup\limits_{t \in [0,T]}\E\Big[ W_1\Big( \frac{1}{N}\sum\limits_{i=1}^N\delta_{X_t^{N,i}},\mathcal{L}(\underbar{X}_t^1) \Big) \Big] \bigg)=0.
  \end{equation}
\end{theorem}

The rate of convergence in \eqref{eq:Thm2ii} is explicitly stated in the next lemma.

\begin{lemma}\label{lem:rates2}
  Supposing the assumptions and notation of Theorem~\ref{thm:propagation2}, it holds that
  \begin{equation}\label{eq:rates2}
    \max\limits_{1\leq i\leq N}\Big( \sup\limits_{t\in [0,T]}\E[|X_t^{N,i}-\underbar{X}_t^i|] \Big) +\sup\limits_{t\in [0,T]}\E\Big[ W_1\Big( \frac{1}{N}\sum\limits_{i=1}^N\delta_{X_t^{N,i}},\mathcal{L}(\underbar{X}_t^1) \Big) \Big] \lesssim N^{-1/2}.
  \end{equation}
\end{lemma}

\begin{remark}
  The rate of convergence in \eqref{eq:rates2} is expected to be optimal for synchronous coupling methods, cf. Remark~\ref{rem:rate of convergence 1}, since it is shown in \cite[Theorem~1 and there after]{Fournier2015} that for terms of the form $\E[W_1(\bar{\rho}_N,\rho)]$ the rate is sharp. Consequently, optimality could be only lost in the inequalities~\eqref{ineq:show0} or \eqref{boundW1}.
\end{remark}

\section{On the well-posedness of ordinary stochastic Volterra equations}\label{sec:well-posedness of SVEs}

In this section, we provide various well-posedness results for ordinary stochastic Volterra equations with random initial conditions that are needed to prove the well-posedness results for mean-field stochastic Volterra equations presented in Section~\ref{sec:main result}. We start with SVEs with Lipschitz continuous coefficients, which is a slight modification of \cite[Theorem~1.1]{Wang2008}.

\begin{lemma}\label{lem:Lipsch_wp}
  Let the kernels $K_\mu, K_\sigma$ fulfill Assumption~\ref{ass:kernels1}, $p>\max\lbrace  \frac{1}{\gamma}, 1+\frac{2}{\epsilon}\rbrace$ with $\gamma\in (0,\frac{1}{2}]$ and $\epsilon>0$ from Assumption~\ref{ass:kernels1}, the initial value $X_0\in L^p(\Omega;\R^d)$ be adapted, and the measurable coefficients $\mu\colon [0,T]\times\R^d\to \R^d$ and $\sigma\colon [0,T]\times\R^d\to \R^{d\times m}$ for some $d,m\in\N$ fulfill the linear growth condition
  \begin{equation*}
    |\mu(t,x)| + |\sigma(t,x)|\leq C_{\mu,\sigma}(1+|x|),
  \end{equation*}
  for some $C_{\mu,\sigma}>0$ and all $t\in[0,T]$, $x\in\R^d$, and the Lipschitz condition
  \begin{equation*}
    |\mu(t,x)-\mu(t,y)| + |\sigma(t,x)-\sigma(t,y)| \leq C_{\mu,\sigma}|x-y|,
  \end{equation*}
  for some $C_{\mu,\sigma}>0$ and all $t\in[0,T]$, $x,y\in\R^d$. Then, the $d$-dimensional stochastic Volterra equation
  \begin{equation*}
    X_t=X_0+\int_0^t K_\mu(s,t)\mu(s,X_s)\dd s+ \int_0^t K_\sigma(s,t)\sigma(s,X_s)\dd B_s,\quad t\in[0,T],
  \end{equation*}
  is well-posed in $L^p$, where $(B_t)_{t\in[0,T]}$ is an $m$-dimensional Brownian motion.
\end{lemma}

\begin{proof}
  With the assumed integrability on $X_0$, it is straightforward to adapt the Picard iteration and the Gr{\"o}nwall type estimates in proof of \cite[Theorem~1.1]{Wang2008} to allow for random initial conditions~$X_0$, as stated in Lemma~\ref{lem:Lipsch_wp}.
\end{proof}

For one-dimensional ordinary stochastic Volterra equations the Lipschitz assumption on the diffusion coefficients can be relaxed to H{\"o}lder continuity, provided the kernels are sufficiently regular or have a convolutional structure. The next results are a slight modification of \cite[Theorem~2.3]{Promel2023}, allowing for SVEs with random initial conditions.

\begin{lemma}\label{lem:Hold_wp}
  Let the kernels $K_\mu, K_\sigma$ fulfill Assumption~\ref{ass:kernels2}, $p>\max\lbrace  \frac{1}{\gamma}, 1+\frac{2}{\epsilon}\rbrace$ with $\gamma\in (0,\frac{1}{2}]$ and $\epsilon>0$ from Assumption~\ref{ass:kernels2}, the initial value $X_0\in L^p(\Omega;\R)$, and the measurable coefficients $\mu\colon [0,T]\times\R\to \R$ and $\sigma\colon [0,T]\times\R\to \R$ fulfill the linear growth condition
  \begin{equation*}
    |\mu(t,x)| + |\sigma(t,x)|\leq C_{\mu,\sigma}(1+|x|),
  \end{equation*}
  for some $C_{\mu,\sigma}>0$ and all $t\in[0,T]$, $x\in\R$, $\mu$ the Lipschitz condition
  \begin{equation*}
    |\mu(t,x)-\mu(t,y)| \leq C_{\mu}|x-y|,
  \end{equation*}
  for some $C_{\mu}>0$ and all $t\in[0,T]$, $x,y\in\R$, and $\sigma$ the H{\"o}lder condition
  \begin{equation*}
    |\sigma(t,x)-\sigma(t,y)| \leq C_{\sigma}|x-y|^{\frac{1}{2}+\xi},
  \end{equation*}
  for $\xi\in[0,\frac{1}{2}]$, some $C_{\sigma}>0$ and all $t\in[0,T]$, $x,y\in\R$. Then, the stochastic Volterra equation
  \begin{equation*}
    X_t=X_0+\int_0^t K_\mu(s,t)\mu(s,X_s)\dd s+ \int_0^t K_\sigma(s,t)\sigma(s,X_s)\dd B_s,\quad t\in[0,T],
  \end{equation*}
  is well-posed in $L^p$, where $(B_t)_{t\in[0,T]}$ is a one-dimensional Brownian motion.
\end{lemma}

\begin{proof}
  With the assumed integrability on $X_0$, it is straightforward to adapt the proof of \cite[Theorem~2.3]{Promel2023} to the case that $X_0$ is a random variable.
\end{proof}

The next lemma is a slight generalization of \cite[Proposition~B.3]{AbiJaberElEuch2019b}, providing the well-posedness of one-dimensional SVEs with convolutional kernels and random initial conditions.

\begin{lemma}\label{lem:well-posednessSVE_convolutional}
  Suppose that $X_0\in L^p(\Omega;\R)$ for some $p>2$, the kernels are of the form $K_\mu(s,t)=K_\sigma(s,t)=\tilde{K}(t-s)$ for some $\tilde{K}\in C^1([0,T];\R)$, and the measurable coefficients $\mu\colon [0,T]\times\R\to \R$ and $\sigma\colon [0,T]\times\R\to \R$ fulfill the linear growth condition
  \begin{equation*}
    |\mu(t,x)| + |\sigma(t,x)|\leq C_{\mu,\sigma}(1+|x|),
  \end{equation*}
  for some $C_{\mu,\sigma}>0$ and all $t\in[0,T]$, $x\in\R$, $\mu$ satisfies the Lipschitz condition
  \begin{equation*}
    |\mu(t,x)-\mu(t,y)| \leq C_{\mu}|x-y|,
  \end{equation*}
  for some $C_{\mu}>0$ and all $t\in[0,T]$, $x,y\in\R$, and $\sigma$ satisfies the H{\"o}lder condition
  \begin{equation*}
    |\sigma(t,x)-\sigma(t,y)| \leq C_{\sigma}|x-y|^{\frac{1}{2}+\xi},
  \end{equation*}
  for $\xi\in[0,\frac{1}{2}]$, some $C_{\sigma}>0$ and all $t\in[0,T]$, $x,y\in\R$. Then, the stochastic Volterra equation
  \begin{equation}
    X_t=X_0+\int_0^t \tilde{K}(t-s)\mu(s,X_s)\dd s+ \int_0^t \tilde{K}(t-s)\sigma(s,X_s)\dd B_s,\quad t\in[0,T],\label{eq:SVE_convolutional}
  \end{equation}
  is well-posed in $L^p$, where $(B_t)_{t\in[0,T]}$ is a one-dimensional Brownian motion.
\end{lemma}

\begin{proof}
  The weak existence of some $L^p$-solution to the SVE~\eqref{eq:SVE_convolutional} follows from \cite[Theorem~3.3]{Promel2022} with the straightforward adaptation to random initial conditions~$X_0$. For the pathwise uniqueness, one can adapt the proof from \cite[Proposition~B.3]{AbiJaberElEuch2019b} using the Lipschitz and H{\"o}lder continuity of $\mu,\sigma$ uniformly in~$t$.
\end{proof}

Moreover, for the well-posedness results of mean-field SVEs we need a multi-dimensional well-posedness result for stochastic Volterra equations where the H{\"o}lder continuous coefficient~$\sigma$ is a diagonal matrix, where each entry only depends on the component of the solution of the respective dimension, as provided in the next remark.

\begin{remark}\label{rem:multi-dimVolterra}
  For $N\in\N$ let us consider the $N$-dimensional stochastic Volterra equation
  \begin{equation}\label{eq:NdimSVE}
    X_t = X_0 +\int_0^t K_\mu(s,t)\mu(s,X_s)\dd s + \int_0^t K_\sigma(s,t)\sigma(s,X_s)\dd B_s,\quad t\in[0,T],
  \end{equation}
  where $(B_t)_{t\in[0,T]}$ is an $N$-dimensional Brownian motion,
  \begin{align*}
    X_t=\left(\begin{array}{c} X_t^1 \\ \vdots \\ X_t^N \end{array}\right),\quad
    X_0=\left(\begin{array}{c} X_0^1 \\ \vdots \\ X_0^N \end{array}\right), \quad
    \mu(s,X_s)=\left(\begin{array}{c} \mu_1(s,X_s) \\ \vdots \\ \mu_N(s,X_s) \end{array}\right)
  \end{align*}
  and
  \begin{align*}
    \sigma(s,X_s)=\left( \begin{array}{ccc}
    \sigma_1(s,X_s^1) &\cdots & 0 \\
    \vdots & \ddots & \vdots \\
    0 & \cdots& \sigma_N(s,X_s^N) \\
    \end{array}\right).
  \end{align*}
  Suppose that the kernels $K_\mu,K_\sigma$ and the initial value $X_0$ fulfill Assumption~\ref{ass:kernels2} or Assumption~\ref{ass:kernels3} with $p$ given from there, that $\mu\colon [0,T]\times \R^N\to \R^N$ is Lipschitz continuous in the space variable, uniformly in the time variable, and each $\sigma_i\colon [0,T]\times\R\to \R$ for $i\in\lbrace1,\dots, N\rbrace$ is $1/2+\xi$-H{\"o}lder continuous in the space variable, uniformly in the time variable for some $\xi\in[0,1/2]$. By considering each dimension separately, as e.g. done for SDEs in \cite[Theorem~1]{Yamada1971}, it is straightforward to conclude the well-posedness in $L^p$ of the SVE~\eqref{eq:NdimSVE} from the corresponding one-dimensional results in Lemma~\ref{lem:Hold_wp} and Lemma~\ref{lem:well-posednessSVE_convolutional}.
\end{remark}

We conclude this section with a remark on the path regularity of solutions and one on the notion of $L^p$-well-posedness.

\begin{remark}[Path regularity]\label{rem:2}
  Let $X$ be the unique ($d$-, $1$- or $N$-dimensional) solution to the stochastic Volterra equation in either of the settings in Lemma~\ref{lem:Lipsch_wp}, Lemma~\ref{lem:Hold_wp}, Lemma~\ref{lem:well-posednessSVE_convolutional} or Remark~\ref{rem:multi-dimVolterra} with $p>\max\lbrace  \frac{1}{\gamma}, 1+\frac{2}{\epsilon}\rbrace$. In the case of Assumption~\ref{ass:kernels3}, we can set $\gamma=\frac{1}{2}$ and $p>2$ given from there. Assuming $X_0\in L^q$ for $q\geq p$, by adapting \cite[Lemma~3.1 and Lemma~3.4]{Promel2023} to the multi-dimensional setting, it follows that
  \begin{equation*}
    \sup\limits_{t\in[0,T]}\E[|X_t|^q]<\infty,
  \end{equation*}
  and
  \begin{equation*}
    \E[|X_t-X_s|^q]\lesssim |t-s|^{\beta q}
  \end{equation*}
  for any $q\geq 1$, $\beta\in (0,\gamma-\frac{1}{p})$, $s,t\in[0,T]$, and, hence, that the solution $X$ has a modification with $\beta$-H{\"o}lder continuous sample paths.
\end{remark}

\begin{remark}
  The notion of $L^p$-well-posedness, as used Lemma~\ref{lem:Lipsch_wp}, Lemma~\ref{lem:Hold_wp}, Lemma~\ref{lem:well-posednessSVE_convolutional} and Remark~\ref{rem:multi-dimVolterra}, appears to be necessary to prove the existence of a strong solution and pathwise uniqueness. First, one needs to assume that a solution $X$ is in $L^p(\Omega\times[0,T];\R^d)$ to conclude continuity of its sample paths with standard estimates, as in \cite[Lemma~3.1]{Promel2023}. Secondly, in order to be able to apply Gr{\"o}nwall's Lemma to an inequality of the form
  \begin{equation*}
    \E[|X_t-Y_t|^p]\lesssim \int_0^t \E[|X_s-Y_s|^p]\dd s,
  \end{equation*}
  one needs to assume that both solutions $X,Y$ are in $L^p(\Omega\times[0,T];\R^d)$ to guarantee finiteness of the expectations $\sup_{s\in[0,t]}\E[|X_s|^p]$ and $\sup_{s\in[0,t]}\E[|Y_s|^p]$ by standard estimates, as in \cite[Lemma~3.4]{Promel2023}.
\end{remark}

\section{Well-posedness: Proof of Theorem~\ref{thm:well-posedness} and~\ref{thm:well-posedness2}}\label{sec:1}

This section is devoted to the proofs of Theorem~\ref{thm:well-posedness} and of Theorem~\ref{thm:well-posedness2}.

\begin{proof}[Proof of Theorem~\ref{thm:well-posedness}]
  We define the solution map~$\Phi$ by
  \begin{equation}\label{eq:solution map}
    \Phi\colon C\big([0,T];\mathcal{P}_{\delta}(\R^d)\big)\to C\big( [0,T];\mathcal{P}_{\delta}(\R^d) \big),
    \quad  \rho \mapsto \Phi(\rho) := \big(\mathcal{L}(X^\rho_t)\big)_{t\in[0,T]},
  \end{equation}
  where $X^\rho$ is the unique $L^p$-solution to the stochastic Volterra equation
  \begin{equation}\label{eq:solution SVE}
    X_t = X_0+\int_0^t K_{\mu}(s,t)\mu(s,X_s,\rho_s)\dd s+\int_0^t K_{\sigma}(s,t)\sigma(s,X_s,\rho_s)\dd B_s,\quad t\in [0,T].
  \end{equation}
  Note that a unique fixed point of the solution map~$\Phi$ implies the existence of a unique $L^p$-solution $X=(X_t)_{t\in[0,T]}$ to the mean-field SVE~\eqref{eq:MVSVE} satisfying $\sup_{t\in[0,T]}\E[|X_t|^{q}]<\infty$ for every $q\geq 1$, c.f. Step~1 below. Hence, it is sufficient to prove that the solution map~$\Phi$ has a unique fixed point.

  \textit{Step~1:} We show the well-definedness of the solution map~$\Phi$.

  For a fixed $\rho=(\rho_t)_{t\in[0,T]}\in C([0,T];\mathcal{P}_{\delta}(\R^d))$, the integral equation~\eqref{eq:solution SVE} is an ordinary stochastic Volterra equation. Due to Assumption~\ref{ass:coefficients1}, the linear growth and Lipschitz condition of Lemma~\ref{lem:Lipsch_wp} are satisfied. Hence, there exists a unique strong $L^p$-solution $X^\rho=(X^\rho_t)_{t\in[0,T]}$ to the SVE~\eqref{eq:solution SVE} and, by Remark~\ref{rem:2}, we get that $\sup_{t\in[0,T]}\E[|X^\rho_t|^q]<\infty$ for $q\geq q$, provided $X_0\in L^q$, and that the sample paths of $X^\rho$ are almost surely continuous. Moreover, note that $\big(\mathcal{L}(X^\rho_t)\big)_{t\in[0,T]}\in C\big( [0,T];\mathcal{P}_{\delta}(\R^d) \big)$, since, by the representation of the Wasserstein distance in terms of random variables (see \cite[(5.14)]{Carmona2018}) and by Remark~\ref{rem:2}, we have
  \begin{equation*}
    W_\delta\big( \mathcal{L}(X_t^\rho),\mathcal{L}(X_s^\rho) \big)
    \leq \E\big[ |X_t^\rho-X_s^\rho |^\delta \big]^{\frac{1}{\delta}}
    \lesssim |t-s|^\beta,\quad  s,t\in[0,T],
  \end{equation*}
  for any $\beta\in (0,\gamma-1/p)$ with $\gamma\in (0,\frac{1}{2}]$, where the parameters are given in Assumption~\ref{ass:kernels1}.

  \textit{Step~2:} For $\rho,\tilde{\rho}\in C\big([0,T];\mathcal{P}_{\delta}(\R^d)\big)$, we show that
  \begin{equation}\label{ineq:W}
    \sup\limits_{s\in[0,t]}W_\delta(\Phi(\rho)_s,\Phi(\tilde{\rho})_s)^\delta\lesssim \int_0^t W_\delta(\rho_s,\tilde{\rho}_s)^\delta \dd s,\quad t\in [0,T].
  \end{equation}

  We get that
  \begin{align}\label{eq:L1bound}
    \E\big[ |X_t^\rho-X_t^{\tilde{\rho}}|^{\delta} \big]
    &\lesssim \E\bigg[ \Big |\int_0^t K_\mu(s,t)\big(  \mu(s,X^\rho_s,\rho_s)-\mu(s,X^{\tilde{\rho}}_s,\tilde{\rho}_s) \big)\dd s \Big|^{\delta} \bigg]\notag\\
    &\qquad +\E\bigg[ \Big|\int_0^t K_\sigma(s,t)\big( \sigma(s,X^\rho_s,\rho_s)-\sigma(s,X^{\tilde{\rho}}_s,\tilde{\rho}_s) \big)\dd B_s \Big|^{\delta} \bigg]\notag\\
    &\lesssim \Big( \int_0^t|K_\mu(s,t)|^{\frac{4+2\epsilon}{4+\epsilon}}\dd s \Big)^{\frac{4+\epsilon}{\epsilon}} \int_0^t \E\big[ \big|\mu(s,X^\rho_s,\rho_s)-\mu(s,X^{\tilde{\rho}}_s,\tilde{\rho}_s) \big|^{\delta}\big]\dd s \notag\\
    &\qquad +\E\bigg[ \Big( \int_0^t \big| K_\sigma(s,t)\big( \sigma(s,X^\rho_s,\rho_s)-\sigma(s,X^{\tilde{\rho}}_s,\tilde{\rho}_s) \big)\big|^2\dd s \Big)^{\frac{\delta}{2}} \bigg]\notag\\
    &\lesssim \int_0^t \E\big[ \big|\mu(s,X^\rho_s,\rho_s)-\mu(s,X^{\tilde{\rho}}_s,\tilde{\rho}_s) \big|^{\delta}\big]\dd s \notag\\
    &\qquad +\Big( \int_0^t|K_\sigma(s,t)|^{2+\epsilon}\dd s \Big)^{\frac{4+2\epsilon}{\epsilon(2+\epsilon)}} \int_0^t \E\big[\big| \sigma(s,X^\rho_s,\rho_s)-\sigma(s,X^{\tilde{\rho}}_s,\tilde{\rho}_s) \big|^\delta\big]\dd s \notag\\
    &\lesssim \int_0^t \Big(\E\big[ \big| X^\rho_s-X^{\tilde{\rho}}_s\big|^{\delta}\big]+W_\delta(\rho_s,\tilde{\rho}_s)^\delta\Big)\dd s.
  \end{align}
  for $t\in[0,T]$, where we used H{\"o}lder's inequality in the drift integral with $\frac{4+2\epsilon}{4+\epsilon}<1+\epsilon$ (noting that $\frac{4+2\epsilon}{4+\epsilon}$ is the conjugate of $\delta/2$) such that, by the choice of $\delta$ in \eqref{eq:delta}, $\frac{4+\epsilon}{4+2\epsilon}+\frac{1}{\delta}=1$ and in the diffusion integral with $\frac{2+\epsilon}{2}$ such that \eqref{eq:delta_Holder} holds, Burkholder--Davis--Gundy's inequality applied to the stochastic processes $(\int_0^r K_\sigma(s,t)\big( \sigma(s,X^\rho_s,\rho_s)-\sigma(s,X^{\tilde{\rho}}_s,\tilde{\rho}_s) \big)\dd B_s)_{r\in [0,t]}$, Fubini's theorem, the integrability of the kernels from Assumption~\ref{ass:kernels1} and the Lipschitz continuity of $\mu$ and $\sigma$ from Assumption~\ref{ass:coefficients1}. Since we have that
  \begin{equation*}
    \sup_{s\in[0,T]}\E[|X_s^\rho-X_s^{\tilde{\rho}}|^\delta]<\infty,
  \end{equation*}
  we can apply Gr{\"o}nwall's inequality to conclude that
  \begin{equation}\label{ineq:Gronwall}
    \E\big[ |X_t^\rho-X_t^{\tilde{\rho}}|^{\delta} \big] \lesssim \int_0^t W_\delta(\rho_s,\tilde{\rho}_s)^\delta \dd s.
  \end{equation}
  Since by assumption $\rho,\tilde{\rho}\in C\big([0,T];\mathcal{P}_{\delta}(\R^d)\big)$, we can bound the Wasserstein distance by
  \begin{equation*}
    W_\delta(\Phi(\rho)_t,\Phi(\tilde{\rho})_t)=W_\delta(\mathcal{L}(X_t^\rho),\mathcal{L}(X_t^{\tilde{\rho}}))\leq \E[|X_t^\rho-X_t^{\tilde{\rho}}|^\delta]^{\frac{1}{\delta}},
  \end{equation*}
  c.f. \cite[(5.14)]{Carmona2018}, and plugging this into \eqref{ineq:Gronwall} and taking the supremum, we obtain~\eqref{ineq:W}.

  \textit{Step~3:} We show that the solution map~$\Phi$ has a unique fixed point.

  First note that it is sufficient to show that $\Phi^k$ is a contraction, see \cite[Theorem]{Bryant1968}, since the Wasserstein space $C([0,T];\mathcal{P}_\delta(\R^d))$ is a complete metric space, see e.g. \cite[Proposition~2.2.8]{Panaretos2020}, where $\Phi^k$ denotes the $k$-th composition of $\Phi$ with itself. Let $C>0$ denote the generic constant in \eqref{ineq:W}. Then, we get iteratively for $k\in\N$,
  \begin{align*}
    \sup\limits_{s\in[0,T]}W_\delta(\Phi^k(\rho)_s,\Phi^k(\tilde{\rho})_s)^\delta &\leq C^k\int_0^T \frac{(T-s)^{k-1}}{(k-1)!}W_\delta(\rho_s,\tilde{\rho}_s)^\delta\dd s\\
    &\leq \frac{C^kT^k}{k!}\sup\limits_{s\in[0,T]} W_\delta(\rho_s,\tilde{\rho}_s)^\delta.
  \end{align*}
  Thus, choosing $k$ large enough such that $\frac{C^kT^k}{k!}<1$, we see that the mapping $\Phi^k$ is a contraction and, hence, $\Phi$ admits a unique fixed point, which completes the proof.
\end{proof}

Next, we provide the proof of Theorem~\ref{thm:well-posedness2}. We keep its presentation fairly short since it is in parts similar to the proof of Theorem~\ref{thm:well-posedness}.

\begin{proof}[Proof of Theorem~\ref{thm:well-posedness2}]
  We again consider the solution map~$\Phi$, as defined in \eqref{eq:solution map}, but choose $\delta=1$ and $d=1$, that is,
  \begin{equation*}
    \Phi\colon C\big([0,T];\mathcal{P}_{1}(\R)\big)\to C\big( [0,T];\mathcal{P}_{1}(\R) \big), \quad  \rho \mapsto \Phi(\rho) := \big(\mathcal{L}(X^\rho_t)\big)_{t\in[0,T]},
  \end{equation*}
   where $X^\rho$ is the unique $L^p$-solution to the stochastic Volterra equation
  \begin{equation*}
    X_t = X_0+\int_0^t K_{\mu}(s,t)\mu(s,X_s,\rho_s)\dd s+\int_0^t K_{\sigma}(s,t)\sigma(s,X_s)\dd B_s,\quad t\in [0,T].
  \end{equation*}
  In the following we show that the solution map $\Phi$ possesses a unique fixed point. We proceed as in the proof of Theorem~\ref{thm:well-posedness}. \textit{Step 1} works exactly the same, using Lemma~\ref{lem:Hold_wp} and Lemma~\ref{lem:well-posednessSVE_convolutional}, respectively, instead of Lemma~\ref{lem:Lipsch_wp}, and \textit{Step 3} works exactly the same. That means we only need to show \textit{Step 2}, or more precisely, estimate~\eqref{ineq:Gronwall} with $\delta=1$. To do that, we treat the cases that Assumption~\ref{ass:kernels2} or that Assumption~\ref{ass:kernels3} holds separately.

  \medskip
  \textit{Case (i):} Suppose the kernels $K_\mu,K_\sigma$ and initial condition $X_0$ satisfy Assumption~\ref{ass:kernels2}.

  To get an analogue estimate as \eqref{ineq:Gronwall}, we use the semimartingale property of a solution $(X_t^\rho)_{t\in[0,T]}$ to \eqref{eq:MVSVE} with fixed $\rho\in C([0,T];\mathcal{P}_1(\R))$ (cf. \cite[Lemma~3.6]{Promel2023} or \cite[Theorem~3.3]{Protter1985}),
  \begin{align*}
    X_t^\rho - X_0 &= \int_0^t K_\sigma(s,s)\sigma(s,X_s^\rho)\dd B_s+ \int_0^t K_\mu(s,s)\mu(s,X_s^\rho,\rho_s) \dd s \\
    &\qquad +  \int_0^t \bigg(\int_0^s \partial_2 K_\mu(u,s)\mu(u,X_u^\rho,\rho_u) \dd u + \int_0^s \partial_2 K_\sigma(u,s)\sigma(u,X_u^\rho)\dd B_u\bigg) \dd s,
  \end{align*}
  and the Yamada--Watanabe functions $\phi_n$ for $n\in\N$ (cf. \cite[Proof of Theorem~5.3]{Promel2023} or the original work \cite{Yamada1971}) that approximate the absolute value function in the following way: Let $(a_n)_{n\in\N}$ be a strictly decreasing sequence with $a_0=1$ such that $a_n\to 0$ as $n\to \infty$ and
  \begin{align*}
    \int_{a_n}^{a_{n-1}}\frac{1}{|x|^{1+2\xi}}\dd x=n,
  \end{align*}
  where $\frac{1}{2}+\xi$ is the H{\"o}lder regularity of $\sigma$.
  Furthermore, we define a sequence of mollifiers: let $(\psi_n)_{n\in\N}\in C_0^{\infty}(\R)$ be smooth functions with compact support such that $\textup{supp}(\psi_n)\subset (a_n,a_{n-1})$, and with the properties
  \begin{align}\label{prop}
    0\leq \psi_n(x)\leq \frac{2}{n|x|^{1+2\xi}}, \quad \forall x\in\R,
    \quad\text{and}\quad
    \int_{a_n}^{a_{n-1}}\psi_n(x)\dd x=1.
  \end{align}
  We set
  \begin{align*}
    \phi_n(x):=\int_0^{|x|}\left(\int_0^y \psi_n(z)\dd z\right)\dd y,\quad x\in \R.
  \end{align*}
  By \eqref{prop} and the compact support of $\psi_n$, it follows that $\phi_n(\cdot)\to |\cdot|$ uniformly as $n\to \infty$. Since every $\psi_n$ and, thus, every $\phi_n$ is zero in a neighborhood around zero, the functions~$\phi_n$ are smooth with
  \begin{equation*}
    \|\phi_n'\|_\infty\leq 1,
    \quad
    \phi_n'(x)=\textup{sgn}(x)\int_0^{|x|}\psi_n(y)\dd y,
    \quad\text{and}\quad
    \phi_n''(x)=\psi_n(|x|)
    \quad\text{for } x\in\R,
  \end{equation*}
  where $\|\cdot\|_\infty$ denotes the $\sup$-norm on $\R$.
 
  Using $\phi_n$, we apply It{\^o}'s formula to $\tilde{X}_t:=X_t^\rho-X_t^{\tilde{\rho}}$, with the notation
  \begin{equation*}
    \tilde{Z}_t:=\int_0^t\big(\mu(s,X^\rho_s,\rho_s)-\mu(s,X^{\tilde{\rho}}_s,\tilde{\rho}_s)\big)\dd s,
    \,\,\, Y_t^\rho:=\int_0^t \sigma(s,X_s^\rho)\dd B_s,
    \,\,\, H_t^\rho:=\int_0^t\partial_2 K_\sigma(s,t)\dd Y_s^\rho,
  \end{equation*}
  and $Y_t^{\tilde{\rho}}$ and $H_t^{\tilde{\rho}}$ analogue, as well as $\tilde{Y}_t:=Y_t^\rho-Y_t^{\tilde{\rho}}$, and $\tilde{H}_t:=H_t^\rho-H_t^{\tilde{\rho}}$, for $t\in [0,T]$, to obtain
  \begin{align}\label{eqmean}
    \phi_n(\tilde{X}_t)
    =&\int_0^t\phi_n'(\tilde{X}_s)\dd\tilde{X}_s+\frac{1}{2}\int_0^t \phi_n''(\tilde{X}_s)\dd\langle \tilde{X}\rangle_s\notag\\
    =&\int_0^t\phi_n'(\tilde{X}_s)K_\mu(s,s)(\mu(s,X_s^\rho,\rho_s)-\mu(s,X_s^{\tilde{\rho}},\tilde{\rho}_s))\dd s\notag\\
    &+ \int_0^t \phi_n'(\tilde{X}_s)\bigg(\int_0^s \partial_2 K_\mu(u,s)\dd \tilde{Z}_u\bigg)\dd s \notag\\
    &+ \int_0^t\phi_n'(\tilde{X}_s)\tilde{H}_s\dd s
    + \int_0^t\phi_n'(\tilde{X}_s)K_\sigma(s,s)\dd \tilde{Y}_s\notag\\
    &+ \frac{1}{2}\int_0^t \phi_n''(\tilde{X}_s) K_\sigma(s,s)^2\big(\sigma(s,X_s^\rho)-\sigma(s,X_s^{\tilde{\rho}})\big)^2 \dd s\notag\\
    =:&I_{1,t}^n+I_{2,t}^n+I_{3,t}^n+I_{4,t}^n+I_{5,t}^n.
  \end{align}
  Let us remark that $H_t^\rho$ and $H_t^{\tilde{\rho}}$ are well-defined stochastic It{\^o} integrals due to Assumption~\ref{ass:kernels2}.

  For $I_{1,t}^n$, the bound $\|\phi'_n\|_\infty \leq 1$, boundedness of~$K_\mu$, Lipschitz continuity of $\mu$, and Jensen's inequality yield
  \begin{equation}\label{i1}
    \E[I_{1,t}^n]\lesssim \int_0^t \big(\E[|\tilde{X}_{s}|]+W_1(\rho_s,\tilde{\rho}_s)\big)\dd s.
  \end{equation}
  For $I_{2,t}^n$, we additionally use the boundedness of $\partial_2 K_\mu(u,s)$ on $\Delta_T$ to obtain
  \begin{equation}\label{i12}
    \E[I_{2,t}^n]\lesssim \int_0^t \big(\E[|\tilde{X}_{s}|]+W_1(\rho_s,\tilde{\rho}_s)\big)\dd s.
  \end{equation}
  For $I_{3,t}^n$, we use $\|\phi'_n\|_\infty \leq 1$ and the integration by parts formula to estimate
  \begin{align}\label{i2}
    \E[I_{3,t}^n]&\leq \int_0^t \E[|\tilde{H}_{s}|]\dd s\notag\\
    &\leq \int_0^t |\partial_2 K_\sigma(s,s)|\E[|\tilde{Y}_{s}|]\dd s + \int_0^t \int_0^s |\partial_{21}K_\sigma(u,s)|\E[|\tilde{Y}_{u}|]\dd u \dd s\notag\\
    &\leq \int_0^t \E[|\tilde{Y}_{s}|] \bigg( \partial_2 K_\sigma(s,s) + \int_s^t |\partial_{21}K_\sigma(s,u)|\dd u \bigg)\dd s\notag\\
    &\lesssim \int_0^t \E[|\tilde{Y}_s|]\dd s,
  \end{align}
  with the boundedness of $\partial_2 K_\sigma(s,s)$ and $\int_s^t\partial_{21}K_\sigma(s,u)\dd u$ from Assumption~\ref{ass:kernels2}. For $I_{4,t}^n$, since $I_{4,t}^n$ is a martingale by \cite[p.~73, Corollary~3]{Protter2004} due to the boundedness of $K_\sigma$, the growth bound on $\sigma$ and the finiteness of the moments of $X^\rho$ and $X^{\tilde{\rho}}$ (cf. \cite[Theorem~2.3]{Promel2023}), we get
  \begin{equation}\label{i3}
    \E[I^n_{4,t}]=\E\left[ \int_0^{t}\phi_n'(\tilde{X}_s)K_\sigma(s,s)(\sigma(s,X_s^\rho)-\sigma(s,X_s^{\tilde{\rho}}))\dd B_s \right]=0.
  \end{equation}
  For $I_{5,t}^n$, we get by using the boundedness of $K_\sigma$, the H{\"o}lder continuity of $\sigma$, and the inequality $\phi_n''(x)\leq \frac{2}{n|x|^{1+2\xi}}$ that
  \begin{equation}\label{i4}
    \E[I_{5,t}^n]\lesssim \E\bigg[\int_0^t \phi_n''(\tilde{X}_{s})|\tilde{X}_{s}|^{1+2\xi}\dd s\bigg]
    \leq \E\bigg[\int_0^t \frac{2}{n|\tilde{X}_{s}|^{1+2\xi}}|\tilde{X}_{s}|^{1+2\xi}\dd s\bigg]
    \lesssim \frac{1}{n}.
  \end{equation}

  Sending $n\to\infty$ and combining the five previous estimates \eqref{i1}, \eqref{i12}, \eqref{i2}, \eqref{i3} and \eqref{i4} with \eqref{eqmean} yields
  \begin{equation}\label{part1}
    \E[|\tilde{X}_{t}|]\lesssim \int_0^t\Big(\E[|\tilde{X}_{s}|] + \E[|\tilde{Y}_{s}|]  + W_1(\rho_s,\tilde{\rho}_s)\Big)\dd s.
  \end{equation}
  To apply Gr{\"o}nwall's lemma, we set $M(t):=\E[|\tilde{X}_{t}|] + \E[|\tilde{Y}_{t}|]$ for $t\in[0,T]$. To find a bound for $\E[|\tilde{Y}_{t}|]$, we apply integration by part formula to obtain
  \begin{align}\label{equiv}
    \tilde{X}_t&=\int_0^t K_\mu(s,t)(\mu(s,X_s^\rho,\rho_s)-\mu(s,X_s^{\tilde{\rho}},\tilde{\rho}_s))\dd s + \int_0^t K_\sigma(s,t)\dd \tilde{Y}_s\notag\\
    &= \int_0^t K_\mu(s,t)(\mu(s,X_s^\rho,\rho_s)-\mu(s,X_s^{\tilde{\rho}},\tilde{\rho}_s))\dd s + K_\sigma(t,t)\tilde{Y}_t-\int_0^t\partial_1 K_\sigma (s,t)\tilde{Y}_s\dd s
  \end{align}
  keeping in mind that that $K_\sigma(\cdot,t)$ is absolutely continuous for every $t\in [0,T]$. Due to $|K_\sigma(t,t)|> C$ for some constant~$C>0$, we can rearrange \eqref{equiv} and use \eqref{part1} to get
  \begin{align}\label{part2i}
    \E\left[|\tilde{Y}_{t}|\right]
    \leq & C\bigg(
    \int_0^t\E\big[|\mu(s,X_s^\rho,\rho_s)-\mu(s,X_s^{\tilde{\rho}},\tilde{\rho}_s)|\big]\dd s \notag\\
    &+ \E\big[|\tilde{X}_{t}|\big] + \int_0^t |\partial_1 K_\sigma(s,t)|\E\big[|\tilde{Y}_{s}|\big]\dd s\bigg)\notag\\
    \lesssim & \int_0^t\Big(\E\big[|\tilde{X}_{s}|\big]+\E[|\tilde{Y}_s|]+W_1(\rho_s,\tilde{\rho}_s)\Big) \dd s.
  \end{align}
  Now, Gr{\"o}nwall's Lemma applied to \eqref{part1} and \eqref{part2i} yields $M(t)\lesssim \int_0^t W_1(\rho_s,\tilde{\rho_s})\dd s$ and hence $\E[|X_t^\rho-X_t^{\tilde{\rho}}|]\lesssim \int_0^t W_1(\rho_s,\tilde{\rho}_s)\dd s$, which is the analogue estimate of \eqref{ineq:Gronwall}.
  
  \medskip
  \textit{Case (ii):} Suppose the kernels $K_\mu,K_\sigma$ and initial condition $X_0$ to satisfy Assumption~\ref{ass:kernels3}. We need to find an analogue to estimate~\eqref{ineq:Gronwall}. By using the notation $\tilde{X}_t:=X_t^\rho-X_t^{\tilde{\rho}}$ and $Y_t^\rho:=\int_0^t \mu(s,X_s^\rho,\rho_s)\dd s+\int_0^t \sigma(s,X_s^\rho)\dd B_s$, $Y^{\tilde{\rho}}_t$ analogue, $\tilde{Y}_t:=Y^{\rho}_t-Y^{\tilde{\rho}}_t$ and the semimartingale property
  \begin{equation*}
    X_t^\rho - X_0 = \int_0^t \tilde{K}(0)\dd Y_s^\rho+ \int_0^t \int_0^s \tilde{K}^\prime(s-u)\dd Y_u^\rho \dd s,
  \end{equation*}
  we can implement the Yamada--Watanabe approach with
  \begin{align}\label{eqmean2}
    \phi_n(\tilde{X}_t)
    =&\int_0^t\phi_n'(\tilde{X}_s)\tilde{K}(0)\dd \tilde{Y}_s + \int_0^t \phi_n'(\tilde{X}_s)\int_0^s \tilde{K}'(s-u)\dd \tilde{Y}_u \dd s\notag\\
     &+ \frac{1}{2}\int_0^t \phi_n''(\tilde{X}_s) \tilde{K}(0)^2\big(\sigma(s,X_s^\rho)-\sigma(s,X_s^{\tilde{\rho}})\big)^2 \dd s\notag\\
    =:&I_{1,t}^n+I_{2,t}^n+I_{3,t}^n.
  \end{align}
  Now, the Lipschitz assumption on $\mu$ applied to $I_{1,t}^n$ and $I_{2,t}^n$ such as the H{\"o}lder assumption on $\sigma$ applied to $I_{3,t}^n$, the boundedness of $\tilde{K}$ and $\tilde{K}^\prime$, the inequalities $\|\phi_n\|_\infty\leq 1$ and $\phi_n^{\prime\prime}(x)\leq \frac{2}{n|x|^{1+2\xi}}$, and sending $n\to\infty$ yields as in \textit{Case (i)} with Gr{\"o}nwall's Lemma the inequality $\E[|X_t^\rho-X_t^{\tilde{\rho}}|]\lesssim \int_0^t W_1(\rho_s,\tilde{\rho}_s)\dd s$, which implies the estimate~\eqref{ineq:Gronwall} and, hence, yields the claimed well-posedness of the mean-field SVE~\eqref{eq:MVSVE2}.
\end{proof}

\begin{remark}\label{rem:1}
  The well-posedness from Theorems~\ref{thm:well-posedness} and \ref{thm:well-posedness2} implies together with a general version of the classical Yamada--Watanabe result (see e.g. \cite[Theorem~1.5]{Kurtz2014}, see also \cite[Example 2.14]{Kurtz2014}) that there is some measurable map $G\colon \R^d\times C([0,T];\R^m)\to C([0,T];\R^d)$ such that any solution $X$ of \eqref{eq:MVSVE} and \eqref{eq:MVSVE2}, respectively, given some initial value $X_0$ and Brownian motion $B$ can be represented as $X=G(X_0,B)$. Hence, if $X,\tilde{X}$ are solutions of \eqref{eq:MVSVE} and \eqref{eq:MVSVE2}, respectively, for initial values $X_0,\tilde{X}_0$ with the same law, and Brownian motions $B,\tilde{B}$, it is straightforward that $\mathcal{L}(X_t)=\mathcal{L}(\tilde{X}_t)$ a.s. for all $t\in[0,T]$.
\end{remark}

\section{Propagation of chaos: Proof of Theorem~\ref{thm:propagation} and \ref{thm:propagation2}}\label{sec:2}

An important argument in the proofs of the propagation of chaos results will be to show that the coupled processes $((X^{N,i},\underbar{X}^i))_{1\leq i\leq N}$ are identically distributed. To that end, the following lemma plays a crucial role. Recall that a sequence of random variables $(\zeta^1,\zeta^2,\dots)$ is called exchangeable if for any $N\in\N$ the vectors $(\zeta^1,\dots,\zeta^N)$ and $(\zeta^{\sigma(1)},\dots,\zeta^{\sigma(N)})$ have the same joint distribution, where $\lbrace\sigma(1),\dots,\sigma(N)\rbrace$ is an arbitrary permutation of $\lbrace 1,\dots,N\rbrace$.

\begin{lemma}\label{lem:exchange}
  Let $(A,\mathcal{F}_A)$ and $(B,\mathcal{F}_B)$ be measurable spaces and let for some fixed $N\in\N$, $(\zeta^1,\dots,\zeta^N)$ be an exchangeable family of $A$-valued random variables. Let $F\colon A\to B$ be a measurable function and define the family of random variables $(X^1,\dots,X^N)$ by $X^i:=F(\zeta^i)$ for $i\in \lbrace 1,\dots,N\rbrace$. Further, let $G\colon A^N\to B^N$ be a measurable function which fulfills the following exchangeability property:
  \begin{equation}\label{eq:exchangebailit}
    (y_1,\dots,y_N)=G((x_1,\dots,x_N))\Rightarrow (y_{\sigma(1)},\dots,y_{\sigma(N)})=G((x_{\sigma(1)},\dots,x_{\sigma(N)})),
  \end{equation}
  for arbitrary $x_1,\dots,x_N\in A$ and any permutation $\lbrace\sigma(1),\dots,\sigma(N)\rbrace$ of $\lbrace 1,\dots,N\rbrace$. Define the family of random variables $(Y^1,\dots,Y^N)$ by
  \begin{equation*}
    (Y^1,\dots,Y^N):=G\big( (\zeta^1,\dots,\zeta^N) \big).
  \end{equation*}
  Then, the coupled family of random variables $((X^i,Y^i))_{1\leq i\leq N}$ is exchangeable.
\end{lemma} 

\begin{proof}
  Let $\lbrace \sigma(1),\dots,\sigma(N)\rbrace$ be an arbitrary permutation of $\lbrace 1,\dots,N\rbrace$. By \eqref{eq:exchangebailit}, we have that
  \begin{align}\label{eq:system1}
    Y^{\sigma(1)}&=G_1\big((\zeta^{\sigma(1)},\zeta^{\sigma(2)},\dots,\zeta^{\sigma(N-1)},\zeta^{\sigma(N)})\big),\notag\\
    Y^{\sigma(2)}&=G_1\big((\zeta^{\sigma(2)},\zeta^{\sigma(3)},\dots,\zeta^{\sigma(N)},\zeta^{\sigma(1)})\big),\notag\\
    &\dots \notag\\
    Y^{\sigma(N)}&=G_1\big((\zeta^{\sigma(N)},\zeta^{\sigma(1)},\dots,\zeta^{\sigma(N-2)},\zeta^{\sigma(N-1)})\big),
  \end{align}
  where $G_1$ denotes the first component of the $N$-dimensional mapping~$G$. Define $W^i:=(X^i,Y^i)$ for $i\in\lbrace 1,\dots,N\rbrace$. Then, by the definition of $X^i$ and \eqref{eq:system1},
  \begin{align}\label{eq:system2}
    W^{\sigma(1)}&=\Big(F(\zeta^{\sigma(1)}),G_1\big((\zeta^{\sigma(1)},\zeta^{\sigma(2)},\dots,\zeta^{\sigma(N-1)},\zeta^{\sigma(N)})\big) \Big),\notag\\
    W^{\sigma(2)}&=\Big(F(\zeta^{\sigma(2)}),G_1\big((\zeta^{\sigma(2)},\zeta^{\sigma(3)},\dots,\zeta^{\sigma(N)},\zeta^{\sigma(1)})\big) \Big),\notag\\
    &\dots\notag\\
    W^{\sigma(N)}&=\Big(F(\zeta^{\sigma(N)}),G_1\big((\zeta^{\sigma(N)},\zeta^{\sigma(1)},\dots,\zeta^{\sigma(N-2)},\zeta^{\sigma(N-1)})\big) \Big).
  \end{align}
  Analogous, we have
  \begin{align}\label{eq:system3}
    W^{1}&=\Big(F(\zeta^{1}),G_1\big((\zeta^{1},\zeta^{2},\dots,\zeta^{N-1},\zeta^{N})\big) \Big),\notag\\
    W^{2}&=\Big(F(\zeta^{2}),G_1\big((\zeta^{2},\zeta^{3},\dots,\zeta^{N},\zeta^{1})\big) \Big),\notag\\
    &\dots\notag\\
    W^{N}&=\Big(F(\zeta^{N}),G_1\big((\zeta^{N},\zeta^{1},\dots,\zeta^{N-2},\zeta^{N-1})\big) \Big).
  \end{align}
  Now, since by assumption $(\zeta^1,\dots,\zeta^N)$ and $(\zeta^{\sigma(1)},\dots,\zeta^{\sigma(N)})$ have the same joint distribution, \eqref{eq:system2} and \eqref{eq:system3} yield that also $(W^1,\dots,W^N)$ and $(W^{\sigma(1)},\dots,W^{\sigma(N)})$ have the same joint distribution which proves the claimed exchangeability.
\end{proof}

We start with the proof of Theorem~\ref{thm:propagation}.

\begin{proof}[Proof of Theorem~\ref{thm:propagation}]
  Let us briefly outline the main steps of the proof:
  \begin{itemize}
    \item[\textit{Step 1:}] We show the existence of the system of processes $(X^{N,i})_{i=1,\dots,N}$ uniquely solving~\eqref{eq:X^n}, for every $N\in\N$.
    \item[\textit{Step 2:}] We prove the inequality
    \begin{equation}\label{ineq:show0}
      \E[|X_t^{N,i}-\underbar{X}_t^i|^\delta] \lesssim \int_0^t \E\Big[ W_\delta\Big(\frac{1}{N}\sum\limits_{j=1}^N \delta_{\underbar{X}_s^j},\mathcal{L}(\underbar{X}_s^i)\Big)^\delta \Big]\dd s, \quad t \in [0,T],
    \end{equation}
    for any $ 1\leq i\leq N$. Recall that $X^{N,i}$ is defined in \eqref{eq:X^n} and $\underbar{X}^i$ is defined as the solution of the mean-field SVE~\eqref{eq:MVSVE} with initial condition~$X_0^i$ and driving Brownian motion~$B^i$.
    \item[\textit{Step 3:}] We prove that the right-hand side of \eqref{ineq:show0} tends to zero.
    \item[\textit{Step 4:}] We show that \textit{Step 2} and \textit{Step 3} imply the statement.
  \end{itemize}

  \medskip
  \textit{Step 1:} By the Lipschitz continuity of $\mu$ and $\sigma$, and the observation that $W_\delta(\bar{\rho}_x^N,\bar{\rho}_y^N)^\delta\leq\frac{1}{N}\sum_{j=1}^N |x_j-y_j|^\delta$ for $x,y\in\R^{N\times d}$ with the notation $\bar{\rho}_x^N=\frac{1}{N}\sum_{j=1}^N\delta_{x_j}\in \mathcal{P}_\delta(\R^d)$, we obtain for every $i\in\lbrace 1,\dots,N\rbrace$ the Lipschitz condition
  \begin{align*}
    \big| \mu(t,x_i,\bar{\rho}_x^N)-\mu(t,y_i,\bar{\rho}_y^N) \big |^\delta + \big| \sigma(t,x_i,\bar{\rho}_x^N)-\sigma(t,y_i,\bar{\rho}_y^N) \big|^\delta
    &\lesssim |x_i-y_i|^\delta+\frac{1}{N}\sum\limits_{j=1}^N |x_j-y_j|^\delta \\
    &\lesssim \|x-y\|_{N\times d}^\delta,
  \end{align*}
  where $\|\cdot\|_{N\times d}$ denotes the row sum norm on $\R^{N\times d}$. With the notation $\tilde{\mu}_i(t,x):=\mu(t,x_i,\bar{\rho}_x^N)$ and $\tilde{\sigma}_i(t,x)$ analogue for any $1\leq i\leq N$, we directly conclude that the growth condition is fulfilled by
  \begin{align*}
    \big| \tilde{\mu}_i(t,x) \big|+\big| \tilde{\sigma}_i(t,x) \big|&\leq \big| \tilde{\mu}_i(t,x)-\tilde{\mu}_i(t,0)\big| + \big| \tilde{\sigma}_i(t,x)-\tilde{\sigma}_i(t,0)\big| +\big| \tilde{\mu}_i(t,0)\big| + \big| \tilde{\sigma}_i(t,0)\big| \\
    &\lesssim \|x\|_{N\times d} +\big| \mu(t,0,\delta_0) \big|+\big| \sigma(t,0,\delta_0) \big|\\
    &\lesssim \|x\|_{N\times d} + C_{\delta_0}\\
    &\lesssim 1+\|x\|_{N\times d},
  \end{align*}
  for all $t\in[0,T]$, $x\in\R^{N\times d}$. Thus, due to the equivalence of all norms on the finite dimensional vector space $\R^{N\times d}$, we can apply the standard Volterra well-posedness result for Lipschitz coefficients from Lemma~\ref{lem:Lipsch_wp} to obtain the system of processes $(X^{N,i})_{i=1,\dots,N,}$ which uniquely solves \eqref{eq:X^n}, for every $N\in\N$.
  
  \medskip
  \textit{Step 2:} We consider the first summand on the left-hand side of \eqref{eq:Thm2}, i.e.~$\E[|X_t^{N,i}-\underbar{X}_t^i|^\delta]$. By using H{\"o}lder's inequality like in \eqref{eq:L1bound}, Fubini's theorem and the Burkholder--Davis--Gundy inequality such as the Lipschitz continuity of $\mu$ and $\sigma$, we can bound, for $1\leq i\leq N$,
  \begin{align}
    &\E[|X_t^{N,i}-\underbar{X}_t^i|^\delta]\notag\\
    &\quad= \E\bigg[ \bigg| \int_0^t K_\mu(s,t)\Big( \mu(s,X_s^{N,i},\bar{\rho}_s^N)-\mu\big(s,\underbar{X}_s^i,\mathcal{L}(\underbar{X}_s^i)\big) \Big)\dd s\notag\\
    &\quad\qquad\qquad +\int_0^t K_\sigma(s,t)\Big( \sigma(s,X_s^{N,i},\bar{\rho}_s^N)-\sigma\big(s,\underbar{X}_s^i,\mathcal{L}(\underbar{X}_s^i)\big) \Big)\dd B_s^i\bigg|^\delta \bigg]\notag\\
    &\quad\lesssim \Big( \int_0^t \big|K_\mu(s,t)\big|^{\frac{4+2\epsilon}{4+\epsilon}}\dd s\Big)^{\frac{4+\epsilon}{\epsilon}}\int_0^t \E\Big[\Big| \mu(s,X_s^{N,i},\bar{\rho}_s^N)-\mu\big(s,\underbar{X}_s^i,\mathcal{L}(\underbar{X}_s^i)\big) \Big|^\delta\Big]\dd s\notag\\
    &\quad\qquad\qquad +\E\bigg[\Big(\int_0^t \big| K_\sigma(s,t)\big( \sigma(s,X_s^{N,i},\bar{\rho}_s^N)-\sigma\big(s,\underbar{X}_s^i,\mathcal{L}(\underbar{X}_s^i)\big) \big|^2\dd s \Big)^{\frac{\delta}{2}}\bigg]\notag\\
    &\quad\lesssim \int_0^t \E\Big[\Big| \mu(s,X_s^{N,i},\bar{\rho}_s^N)-\mu\big(s,\underbar{X}_s^i,\mathcal{L}(\underbar{X}_s^i)\big) \Big|^\delta\Big]\dd s\notag\\
    &\quad\qquad +\Big( \int_0^t |K_\sigma(s,t)|^{2+\epsilon}\dd s \Big)^{\frac{4}{\epsilon}} \int_0^t\E\Big[\big| \sigma(s,X_s^{N,i},\bar{\rho}_s^N)-\sigma\big(s,\underbar{X}_s^i,\mathcal{L}(\underbar{X}_s^i)\big) \big|^\delta\Big]\dd s\notag\\
    &\quad\lesssim \int_0^t \E\Big[|X_s^{N,i}-\underbar{X}_s^i|^\delta + W_\delta(\bar{\rho}_s^N,\mathcal{L}(\underbar{X}_s^i))^\delta \Big]\dd s,\label{ineq1}
  \end{align}
  for any $t\in[0,T]$. By Remark~\ref{rem:1}, we obtain that $\mathcal{L}(\underbar{X}_s^i)=\mathcal{L}(\underbar{X}_s^1)$. Hence, we get that
  \begin{align}
    W_\delta\big(\bar{\rho}_s^N,\mathcal{L}(\underbar{X}_s^i)\big)^\delta
    &=W_\delta\bigg(\bar{\rho}_s^N,\mathcal{L}(\underbar{X}_s^1)\bigg)^\delta\notag\\
    &\leq 2^\delta W_\delta\bigg(\bar{\rho}_s^N,\frac{1}{N}\sum\limits_{j=1}^N \delta_{\underbar{X}_s^j}\bigg)^\delta +2^\delta W_\delta\bigg(\frac{1}{N}\sum\limits_{j=1}^N \delta_{\underbar{X}_s^j},\mathcal{L}(\underbar{X}_s^1)\bigg)^\delta\notag\\
    &\lesssim \frac{1}{N} \sum\limits_{j=1}^N \big| X_s^{N,j}-\underbar{X}_s^j \big|^\delta + W_\delta\bigg(\frac{1}{N}\sum\limits_{j=1}^N \delta_{\underbar{X}_s^j},\mathcal{L}(\underbar{X}_s^1)\bigg)^\delta. \label{ineq2}
  \end{align}
  Moreover, by Remark~\ref{rem:1}, we can find a measurable map $G\colon \R^d\times C([0,T];\R^m)\to C([0,T];\R^d)$ such that, for any $1\leq i\leq N$,
  \begin{equation*}
    \underbar{X}^i=G(X_0^i,B^i).
  \end{equation*}
  In the same way, there is a measurable map $G_N\colon (\R^{d}\times C([0,T];\R^m))^N\to C([0,T];\R^d)^N$, such that
  \begin{equation*}
    (X^{N,1},\dots ,X^{N,N})=G_N\big((X_0^1,\dots ,X_0^N),(B^1,\dots ,B^N)\big).
  \end{equation*}
  More generally, by the symmetry of the system \eqref{eq:X^n}, for any permutation $\varsigma$ of $\lbrace 1,\dots ,N \rbrace$, it is
  \begin{equation*}
    (X^{N,\varsigma(1)},\dots ,X^{N,\varsigma(N)})=G_N\big((X_0^{\varsigma(1)},\dots ,X_0^{\varsigma(N)}),(B^{\varsigma(1)},\dots ,B^{\varsigma(N)})\big).
  \end{equation*}
  Hence, since the random variables $((X_0^i,B^i))_{1\leq i\leq N}$ are i.i.d. and, in particular, exchangeable, we can apply Lemma~\ref{lem:exchange} to obtain that the coupled processes $\big((X^{N,i},\underbar{X}^i)\big)_{1\leq i\leq N}$ are exchangeable and hence, in particular, are identically distributed. We can for $i=1$ insert \eqref{ineq2} into \eqref{ineq1} and conclude by Jensen's inequality that
  \begin{align*}
    &\E[|X_t^{N,1}-\underbar{X}_t^1|^\delta]\\
    &\quad\lesssim \int_0^t \E\bigg[|X_s^{N,1}-\underbar{X}_s^1|^\delta + \frac{1}{N} \sum\limits_{j=1}^N \big |X_s^{N,j}-\underbar{X}_s^j \big|^\delta + W_\delta\bigg(\frac{1}{N}\sum\limits_{j=1}^N \delta_{\underbar{X}_s^j},\mathcal{L}(\underbar{X}_s^1)\bigg)^\delta \bigg]\dd s\\
    &\quad= \int_0^t \E\bigg[2|X_s^{N,1}-\underbar{X}_s^1|^\delta + W_\delta\bigg(\frac{1}{N}\sum\limits_{j=1}^N \delta_{\underbar{X}_s^j},\mathcal{L}(\underbar{X}_s^1)\bigg)^\delta \bigg]\dd s.
  \end{align*}
  Using Gr{\"o}nwall's lemma, we deduce
  \begin{equation*}
    \E[|X_t^{N,1}-\underbar{X}_t^1|^\delta] \lesssim \int_0^t \E\bigg[ W_\delta\bigg(\frac{1}{N}\sum\limits_{j=1}^N \delta_{\underbar{X}_s^j},\mathcal{L}(\underbar{X}_s^1)\bigg)^\delta \bigg]\dd s,
  \end{equation*}
  and since the processes $\big((X^{N,i},\underbar{X}^i)\big)_{1\leq i\leq N}$ are identically distributed, this completes \textit{Step~2}.
  
  \medskip
  \textit{Step 3:}
  First, we show that
  \begin{equation}\label{convergence_uniform}
    \lim\limits_{N\to\infty}\E\Big[ W_\delta\Big(\frac{1}{N}\sum\limits_{j=1}^N \delta_{\underbar{X}_s^j},\mathcal{L}(\underbar{X}_s^1)\Big)^\delta \Big]=0,
  \end{equation}
 for any $s\in[0,T]$ by showing convergence in probability and uniform integrability. By the Glivenko--Cantelli theorem (see \cite[Chapter~26, Theorem~1]{Shorack2009} for a general version) and since the $\underbar{X}^j$ are i.i.d., we get the convergence
  \begin{equation*}
    \frac{1}{N}\sum\limits_{j=1}^N \delta_{\underbar{X}_s^j}\to \mathcal{L}(\underbar{X}_s^1), \quad \text{as }N\to \infty,
  \end{equation*}
  almost surely and hence in probability. Furthermore, using again the notation $\bar{\rho}^N_s= \frac{1}{N}\sum\limits_{j=1}^N \delta_{\underbar{X}_s^j}$ we can bound using H{\"o}lder's inequality and the boundedness of all moments of $\underbar{X}_s^i$, $1\leq i\leq N$, in \eqref{fin_moments}, that
  \begin{align}
    &\sup\limits_{N\in\N}\E\big[ W_\delta\big( \bar{\rho}^N_s,\mathcal{L}(\underbar{X}_s^1) \big)^\delta \mathbbm{1}_{\lbrace W_\delta(\bar{\rho}^N_s,\mathcal{L}(\underbar{X}_s^1))>K \rbrace} \big]\notag\\
    &\qquad \leq K^{-1} \sup\limits_{N\in\N} \E\big[ W_\delta\big( \bar{\rho}^N_s,\mathcal{L}(\underbar{X}_s^1) \big)^{\delta+1} \big]\notag\\
    &\qquad \leq K^{-1} \sup\limits_{N\in\N} \E\big[ W_{\delta+1}\big( \bar{\rho}^N_s,\mathcal{L}(\underbar{X}_s^1) \big)^{\delta+1} \big]\notag\\
    &\qquad \leq K^{-1} \sup\limits_{N\in\N}\E\big[ W_{\delta+1}\big( \bar{\rho}^N_s,\delta_0 \big)^{\delta+1} + W_{\delta+1}\big( \delta_0,\mathcal{L}(\underbar{X}_s^1) \big)^{\delta+1} \big]\notag\\
    &\qquad = K^{-1} \sup\limits_{N\in\N}\E\Big[ \frac{1}{N}\big(\sum\limits_{i=1}^N |\underbar{X}_s^i|^{\delta+1}\big) + |\underbar{X}_s^1|^{\delta+1} \Big]\notag\\
    &\qquad = 2K^{-1}\E[|\underbar{X}_s^1|^{\delta+1}]\to 0,\label{eq:konv}
  \end{align}
  as $K\to\infty$, which shows uniform $\delta$-integrability of the family of random variables
  \begin{equation*}
    \bigg(W_\delta\Big(\frac{1}{N}\sum\limits_{j=1}^N \delta_{\underbar{X}_s^j},\mathcal{L}(\underbar{X}_s^1)\Big)\bigg)_{N\in\N}.
  \end{equation*}
  Hence, Vitali's convergence theorem (see \cite[Theorem~4.5.4]{Bogachev2007}) reveals the $L^\delta$-convergence as claimed in \eqref{convergence_uniform}.

  To conclude \textit{Step~3}, it remains to show that the convergence \eqref{convergence_uniform} is uniform in $s$. Therefore, we first notice that for any
  $p\geq \delta$,
  \begin{align}
    \E\bigg[ W_\delta\Big( \frac{1}{N}\sum\limits_{j=1}^N \delta_{\underbar{X}_s^j},\mathcal{L}(\underbar{X}_s^1) \Big)^{p} \bigg]&\leq\E\bigg[ W_p\Big( \frac{1}{N}\sum\limits_{j=1}^N \delta_{\underbar{X}_s^j},\mathcal{L}(\underbar{X}_s^1) \Big)^{p} \bigg]\notag\\
    &\lesssim \E\bigg[ W_p\Big( \frac{1}{N}\sum\limits_{j=1}^N \delta_{\underbar{X}_s^j},\delta_0 \Big)^{p} \bigg] + W_p\big( \delta_0,\mathcal{L}(\underbar{X}_s^1) \big)^{p} \notag\\
    &= \frac{1}{N}\E\Big[ \sum\limits_{j=1}^N |\underbar{X}_s^j|^{p}\Big] + \E[|\underbar{X}_s^1|^{p}] = 2\E[|\underbar{X}_s^1|^{p}] <\infty,\label{eq:boundEW2}
  \end{align}
  by \eqref{fin_moments}. With Jensen's inequality \eqref{eq:boundEW2} also follows for $1\leq p< \delta$.
   
  Let $k :=\lceil \delta\rceil\geq\delta$ denote the smallest integer greater or equal to $\delta$. Notice that with the same argumentation as in \eqref{eq:konv} by substituting the exponent $\delta$ by $k$ and then bounding from above using the $k+1$-Wasserstein distance and again by Vitali's convergence theorem also the $L^k$-convergence of the $\delta$-Wasserstein distance in \eqref{convergence_uniform} follows. Once we show that this $L^k$-convergence is uniform in $s$, then it will follow that
  \begin{align}\label{conv_unif}
    & \lim\limits_{N\to\infty}\sup\limits_{ s\in [0,T]}\E\Big[ W_\delta\Big(\frac{1}{N}\sum\limits_{j=1}^N \delta_{\underbar{X}_s^j},\mathcal{L}(\underbar{X}_s^1)\Big)^\delta \Big]\notag\\
    &\qquad\leq \lim\limits_{N\to\infty}\sup\limits_{ s\in [0,T]}\E\Big[ W_\delta\Big(\frac{1}{N}\sum\limits_{j=1}^N \delta_{\underbar{X}_s^j},\mathcal{L}(\underbar{X}_s^1)\Big)^k \Big]^{\frac{\delta}{k}}=0.
  \end{align}
  Therefore, using the factorization
  \begin{equation}\label{ineq:factor}
    a^{k}-b^{k} =(a-b) \sum_{r=0}^{k-1} a^{k-1-r}b^r,
  \end{equation}
  and H{\"o}lder's inequality with $\delta$ and $q=\frac{4+2\epsilon}{4+\epsilon}$ such that $\frac{1}{\delta}+\frac{1}{q}=1$, we get
  \begin{align}
    &\bigg| \E\Big[ W_\delta\Big( \frac{1}{N}\sum\limits_{j=1}^N \delta_{\underbar{X}_t^j},\mathcal{L}(\underbar{X}_t^1) \Big)^{k} \Big] - \E\Big[ W_\delta\Big( \frac{1}{N}\sum\limits_{j=1}^N \delta_{\underbar{X}_s^j},\mathcal{L}(\underbar{X}_s^1) \Big)^{k} \Big] \bigg|\notag\\
    &\qquad = \bigg|\E\bigg[ \bigg(W_\delta\Big( \frac{1}{N}\sum\limits_{j=1}^N \delta_{\underbar{X}_t^j},\mathcal{L}(\underbar{X}_t^1) \Big)-W_\delta\Big( \frac{1}{N}\sum\limits_{j=1}^N \delta_{\underbar{X}_s^j},\mathcal{L}(\underbar{X}_s^1) \Big)\bigg) \notag \\
    &\qquad\qquad\sum\limits_{r=0}^{k-1}W_\delta\Big( \frac{1}{N}\sum\limits_{j=1}^N \delta_{\underbar{X}_t^j},\mathcal{L}(\underbar{X}_t^1) \Big)^{k-1-r}W_\delta\Big( \frac{1}{N}\sum\limits_{j=1}^N \delta_{\underbar{X}_s^j},\mathcal{L}(\underbar{X}_s^1) \Big)^{r} \bigg]\bigg|\notag\\
    &\qquad \leq \bigg|\E\bigg[ \bigg(W_\delta\Big( \frac{1}{N}\sum\limits_{j=1}^N \delta_{\underbar{X}_t^j},\mathcal{L}(\underbar{X}_t^1) \Big)-W_\delta\Big( \frac{1}{N}\sum\limits_{j=1}^N \delta_{\underbar{X}_s^j},\mathcal{L}(\underbar{X}_s^1) \Big)\bigg)^\delta\bigg]^{\frac{1}{\delta}}  \notag\\
    &\qquad\qquad\E\bigg[\bigg(\sum\limits_{r=0}^{k-1}W_\delta\Big( \frac{1}{N}\sum\limits_{j=1}^N \delta_{\underbar{X}_t^j},\mathcal{L}(\underbar{X}_t^1) \Big)^{k-1-r}W_\delta\Big( \frac{1}{N}\sum\limits_{j=1}^N \delta_{\underbar{X}_s^j},\mathcal{L}(\underbar{X}_s^1) \Big)^{r} \bigg)^q\bigg]^{\frac{1}{q}}\bigg|.\label{ineq:abs1}
  \end{align}
  Using again H{\"o}lder's inequality such as \eqref{eq:boundEW2}, we can bound the second expectation by
  \begin{align}
    &\E\bigg[\bigg(\sum\limits_{r=0}^{k-1}W_\delta\Big( \frac{1}{N}\sum\limits_{j=1}^N \delta_{\underbar{X}_t^j},\mathcal{L}(\underbar{X}_t^1) \Big)^{k-1-r}W_\delta\Big( \frac{1}{N}\sum\limits_{j=1}^N \delta_{\underbar{X}_s^j},\mathcal{L}(\underbar{X}_s^1) \Big)^{r} \bigg)^q\bigg]^{\frac{1}{q}}\notag\\
    &\lesssim \bigg(\sum\limits_{r=0}^{k-1}\E\Big[W_\delta\Big( \frac{1}{N}\sum\limits_{j=1}^N \delta_{\underbar{X}_t^j},\mathcal{L}(\underbar{X}_t^1) \Big)^{q(k-1-r)}W_\delta\Big( \frac{1}{N}\sum\limits_{j=1}^N \delta_{\underbar{X}_s^j},\mathcal{L}(\underbar{X}_s^1) \Big)^{qr}\Big]\bigg)^{\frac{1}{q}}\notag\\
    &\lesssim \bigg(\sum\limits_{r=0}^{k-1}\E\Big[W_\delta\Big( \frac{1}{N}\sum\limits_{j=1}^N \delta_{\underbar{X}_t^j},\mathcal{L}(\underbar{X}_t^1) \Big)^{2q(k-1-r)}\Big]^{\frac{1}{2}}\E\Big[W_\delta\Big( \frac{1}{N}\sum\limits_{j=1}^N \delta_{\underbar{X}_s^j},\mathcal{L}(\underbar{X}_s^1) \Big)^{2qr}\Big]^{\frac{1}{2}}\bigg)^{\frac{1}{q}}<\infty.\label{ineq:abs2}
  \end{align}
  Inserting \eqref{ineq:abs2} into \eqref{ineq:abs1} and using the triangle inequality
  \begin{align*}
    &W_\delta\Big( \frac{1}{N}\sum\limits_{j=1}^N \delta_{\underbar{X}_t^j},\mathcal{L}(\underbar{X}_t^1) \Big)\\
    &\quad\leq W_\delta\Big( \frac{1}{N}\sum\limits_{j=1}^N \delta_{\underbar{X}_t^j},\frac{1}{N}\sum\limits_{j=1}^N \delta_{\underbar{X}_s^j} \Big)+W_\delta\Big( \frac{1}{N}\sum\limits_{j=1}^N \delta_{\underbar{X}_s^j},\mathcal{L}(\underbar{X}_s^1) \Big) +  W_\delta\Big( \mathcal{L}(\underbar{X}_s^1),\mathcal{L}(\underbar{X}_t^1) \Big),
  \end{align*}
  which also holds if we switch $s$ and $t$, we arrive at
  \begin{align*}
    &\bigg| \E\Big[ W_\delta\Big( \frac{1}{N}\sum\limits_{j=1}^N \delta_{\underbar{X}_t^j},\mathcal{L}(\underbar{X}_t^1) \Big)^{k} \Big] - \E\Big[ W_\delta\Big( \frac{1}{N}\sum\limits_{j=1}^N \delta_{\underbar{X}_s^j},\mathcal{L}(\underbar{X}_s^1) \Big)^{k} \Big] \bigg|\\
    &\qquad \lesssim \E\bigg[ \bigg| \bigg(W_\delta\Big( \frac{1}{N}\sum\limits_{j=1}^N \delta_{\underbar{X}_t^j},\mathcal{L}(\underbar{X}_t^1) \Big)-W_\delta\Big( \frac{1}{N}\sum\limits_{j=1}^N \delta_{\underbar{X}_s^j},\mathcal{L}(\underbar{X}_s^1) \Big)\bigg)\bigg|^\delta\bigg]^{\frac{1}{\delta}}\\
    &\qquad \lesssim \E\bigg[ \bigg( W_\delta\Big( \frac{1}{N}\sum\limits_{j=1}^N \delta_{\underbar{X}_t^j},\frac{1}{N}\sum\limits_{j=1}^N \delta_{\underbar{X}_s^j} \Big) +  W_\delta\Big( \mathcal{L}(\underbar{X}_t^1),\mathcal{L}(\underbar{X}_s^1) \Big)\bigg)^\delta \bigg]^{\frac{1}{\delta}}\\
    &\qquad \lesssim \E\Big[ W_\delta\Big( \frac{1}{N}\sum\limits_{j=1}^N \delta_{\underbar{X}_t^j},\frac{1}{N}\sum\limits_{j=1}^N \delta_{\underbar{X}_s^j} \Big)^\delta \Big]^{\frac{1}{\delta}} +\E\Big[  W_\delta\Big( \mathcal{L}(\underbar{X}_t^1),\mathcal{L}(\underbar{X}_s^1) \Big)^\delta \Big]^{\frac{1}{\delta}}\\
    &\qquad \lesssim \E[|\underbar{X}_t^1-\underbar{X}_s^1|^\delta]^{\frac{1}{\delta}}\\
    &\qquad \lesssim |t-s|^\beta,
  \end{align*}
  where the last line holds by Remark~\ref{rem:2} for any $\beta\in (0,\gamma-1/p)$ with $\gamma\in(0,\frac{1}{2}]$ from Assumption~\ref{ass:kernels1}. Hence, we obtain that \eqref{conv_unif} holds which shows together with \eqref{ineq:show0} that
  \begin{equation}
    \lim\limits_{N\to\infty}\sup\limits_{t\in [0,T]}\E[|X_t^{N,1}-\underbar{X}_t^1|^\delta]=0,\label{limit0_2}
  \end{equation}
  and knowing that $((X^{N,i},\underbar{X}^i))_{1\leq i\leq N}$ are identically distributed, this completes \textit{Step~3}.
  
  \medskip
  \textit{Step~4:} We already know from \textit{Step~3} that the first summand in \eqref{eq:Thm2} converges to zero. For the second summand, we use the triangle inequality and Jensen's inequality, to obtain
  \begin{align}
    &\sup\limits_{t\in [0,T]}\E\Big[ W_\delta\Big( \frac{1}{N}\sum\limits_{i=1}^N\delta_{X_t^{N,i}},\mathcal{L}(\underbar{X}_t^1) \Big)^\delta \Big]\notag\\
    &\qquad \lesssim \sup\limits_{ t\in [0,T]}\E\Big[ W_\delta\Big( \frac{1}{N}\sum\limits_{i=1}^N\delta_{X_t^{N,i}},\frac{1}{N}\sum\limits_{i=1}^N\delta_{\underbar{X}_t^i} \Big)^\delta \Big]+\sup\limits_{t\in [0,T]}\E\Big[ W_\delta\Big( \frac{1}{N}\sum\limits_{i=1}^N\delta_{\underbar{X}_t^i},\mathcal{L}(\underbar{X}_t^1) \Big)^\delta \Big]\notag\\
    &\qquad \lesssim \sup\limits_{t\in [0,T]}\E\Big[ \frac{1}{N}\sum\limits_{i=1}^N |X_t^{N,i}-\underbar{X}_t^i|^\delta\Big]+\sup\limits_{t\in [0,T]}\E\Big[ W_\delta\Big( \frac{1}{N}\sum\limits_{i=1}^N\delta_{\underbar{X}_t^i},\mathcal{L}(\underbar{X}_t^1) \Big)^\delta \Big],\label{boundW1}
  \end{align}
  which also tends to $0$ as $N\to\infty$, by \eqref{conv_unif} and \eqref{limit0_2}.
\end{proof}

We continue with the proof of Theorem~\ref{thm:propagation2}. Since the proof is similar to the proof of Theorem~\ref{thm:propagation2}, we focus, for the sake of brevity, on the main differences.

\begin{proof}[Proof of Theorem~\ref{thm:propagation2}]
  We prove the statement by using the same \textit{Step~1}-\textit{Step~4} as in the proof of Theorem~\ref{thm:propagation}, but with $\delta=1$. Only for \textit{Step~2}, we need to differ between Assumption~\ref{ass:kernels2}, referred to as \textit{Case (i)}, and Assumption~\ref{ass:kernels3}, referred to as \textit{Case (ii)}, to hold.

  \medskip
  \textit{Step 1:} By Remark~\ref{rem:multi-dimVolterra}, we obtain as in the proof of Theorem~\ref{thm:propagation} the unique system of stochastic processes $(X^{N,i})_{i=1,\dots,N}$ that solves \eqref{eq:X^n}.
  
  \medskip
  \textit{Step 2:}

  \textit{Case (i):} Suppose the kernels $K_\mu,K_\sigma$ and initial condition $X_0$ satisfy Assumption~\ref{ass:kernels2}. To mimic inequality \eqref{ineq:show0}, we use the semimartingale property
  \begin{align*}
    X_t^{N,i} - \underbar{X}_t^i &= \int_0^t K_\sigma(s,s)\Big(\sigma(s,X_s^{N,i})-\sigma(s, \underbar{X}_s^i)\Big)\dd B_s\\
    &\qquad+ \int_0^t K_\mu(s,s)\Big(\mu(s,X_s^{N,i},\bar{\rho}_s^N)-\mu(s,\underbar{X}_s^i,\mathcal{L}(\underbar{X}_s^i))\Big) \dd s\\
    &\qquad  +  \int_0^t \bigg(\int_0^s \partial_2 K_\mu(u,s)\Big(\mu(u,X_u^{N,i},\bar{\rho}_u^N)-\mu(u,\underbar{X}_u^i,\mathcal{L}(\underbar{X}_u^i))\Big) \dd u \\
    &\qquad+ \int_0^s \partial_2 K_\sigma(u,s)\Big(\sigma(u,X_u^{N,i})-\sigma(u,\underbar{X}_u^i)\Big)\dd B_u\bigg) \dd s,
  \end{align*}
  to perform a Yamada--Watanabe approach exactly as we did around equality~\eqref{eqmean}, and obtain for fixed $i\in\lbrace1,\dots,N\rbrace$ with the notation $M^{N,i}(t):=\E[|X_t^{N,i}-\underbar{X}_t^i|]+\E[|\tilde{Y}_t|]$, where $\tilde{Y}_t:=\int_0^t\sigma(s,X_s^{N,i})\dd B_s^i-\int_0^t \sigma(s,\underbar{X}_s^i)\dd B_s^i$, that
  \begin{equation*}
    M^{N,i}(t)\lesssim \int_0^t \Big(M^{N,i}(s)+\E[W_1(\bar{\rho}_s^N,\mathcal{L}(\underbar{X}^i_s))]\Big)\dd s,
  \end{equation*}
  such that, proceeding as in the proof of Theorem~\ref{thm:propagation} including applying Gr{\"o}nwall's inequality, we obtain
  \begin{equation}\label{ineq:convRHS}
    \E[|X_t^{N,i}-\underbar{X}_t^i|]\lesssim \int_0^t \E\big[ W_1(\frac{1}{N}\sum\limits_{j=1}^N\delta_{\underbar{X}_s^j},\mathcal{L}(\underbar{X}_s^i)) \big]\dd s.
  \end{equation}
  
  \textit{Case (ii):} Suppose the kernels $K_\mu,K_\sigma$ and initial condition $X_0$ satisfy Assumption~\ref{ass:kernels3}. As in \textit{Case (i)} to mimic inequality \eqref{ineq:show0}, we use the semimartingale property 
  \begin{align*}
    X_t^{N,i}-\underbar{X}_t^i &= \int_0^t \tilde{K}(0)\Big(\sigma(s,X_s^{N,i})-\sigma(s, \underbar{X}_s^i)\Big)\dd B_s\\
    &\qquad+ \int_0^t \tilde{K}(0)\Big(\mu(s,X_s^{N,i},\bar{\rho}_s^N)-\mu(s,\underbar{X}_s^i,\mathcal{L}(\underbar{X}_s^i))\Big) \dd s\\
    &\qquad  +  \int_0^t \bigg(\int_0^s \tilde{K}^\prime(s-u)\Big(\mu(u,X_u^{N,i},\bar{\rho}_u^N)-\mu(u,\underbar{X}_u^i,\mathcal{L}(\underbar{X}_u^i))\Big) \dd u \\
    &\qquad+ \int_0^s \tilde{K}^\prime(s-u)\Big(\sigma(u,X_u^{N,i})-\sigma(u,\underbar{X}_u^i)\Big)\dd B_u\bigg) \dd s,
  \end{align*} 
  perform a Yamada--Watanabe approach and apply Gr{\"o}nwall's inequality as in \eqref{eqmean2} which yields
  \begin{equation*}
    \E[|X_t^{N,i}-\underbar{X}_t^i|] \lesssim \int_0^t \E[W_1(\frac{1}{N}\sum\limits_{j=1}^N\delta_{\underbar{X}_s^j},\mathcal{L}(\underbar{X}^i_s))]\Big)\dd s.
  \end{equation*}

  \medskip
  \textit{Step 3:} Obtaining the convergence to zero uniformly in $s$ of the right-hand side of \eqref{ineq:convRHS} follows now easily by using
  \begin{equation*}
    \E\big[ W_1(\frac{1}{N}\sum\limits_{j=1}^N\delta_{\underbar{X}_s^j},\mathcal{L}(\underbar{X}_s^1)) \big]\leq \E\big[ W_2(\frac{1}{N}\sum\limits_{j=1}^N\delta_{\underbar{X}_s^j},\mathcal{L}(\underbar{X}_s^1))^2 \big]^{\frac{1}{2}},
  \end{equation*}
  and then using \cite[(5.19)]{Carmona2018}, and proceeding as in \cite[Proof of Theorem~2.12]{Carmona2018b}. 
  
  \medskip
  \textit{Step 4:} As in \eqref{boundW1}, we obtain
  \begin{align}
    &\sup\limits_{0\leq t\leq T}\E\Big[ W_1\Big( \frac{1}{N}\sum\limits_{i=1}^N\delta_{X_t^{N,i}},\mathcal{L}(\underbar{X}_t^1) \Big) \Big]\notag\\
    &\qquad \lesssim \sup\limits_{0\leq t\leq T}\E\Big[ \frac{1}{N}\sum\limits_{i=1}^N |X_t^{N,i}-\underbar{X}_t^i|\Big]+\sup\limits_{0\leq t\leq T}\E\Big[ W_1\Big( \frac{1}{N}\sum\limits_{i=1}^N\delta_{\underbar{X}_t^i},\mathcal{L}(\underbar{X}_t^1) \Big) \Big],\label{ineq:settingII2}
  \end{align}
  which tends to zero by the uniform convergence to zero of the right-hand side of \eqref{ineq:convRHS}, and finishes the proof.
\end{proof}

\section{Rate of convergence: Proof of Lemma~\ref{lem:rates} and~\ref{lem:rates2}}\label{sec:3}

The proofs of Lemma~\ref{lem:rates} and Lemma~\ref{lem:rates2} rely on a quantitative Glivenko--Cantelli theorem due to Fournier and Guillin~\cite{Fournier2015}, which provides a sharp estimate of the $\delta$-Wasserstein distance. For the sake of completeness, we recall \cite[Theorem~1]{Fournier2015} in the following lemma.

\begin{lemma}\label{lem1}
  Let $\delta>0$ and $\bar{\rho}^N:=\frac{1}{N}\sum\limits_{i=1}^N \delta_{X^i}$ be the empirical distribution of i.i.d. random variables $(X^i)_{i=1,\dots,N}$ with common distribution $\rho$ such that $\rho\in\mathcal{P}_p(\R^d)$ for every $p\geq 1$. Then, we have
  \begin{equation*}
    \E[W_\delta(\bar{\rho}^N,\rho)^\delta]\lesssim \varepsilon_N,
  \end{equation*}
  where $(\varepsilon_N)_{N\in\N}$ is given by \eqref{def:varepsilon}, i.e.
  \begin{equation*}
    \varepsilon_N= \Bigg\{\begin{array}{lr}
        N^{-1/2}, & \text{if } d<2\delta,\\
        N^{-1/2}\log_2(1+N), & \text{if } d=2\delta,\\
        N^{-\delta/d}, & \text{if } d>2\delta,
    \end{array}
  \end{equation*}
  and
  \begin{equation*}
    \E[W_1(\bar{\rho}^N,\rho)]\lesssim N^{-1/2}.
  \end{equation*}
\end{lemma}

With this lemma at hand, we can prove Lemma~\ref{lem:rates} and \ref{lem:rates2}.

\begin{proof}[Proof of Lemma~\ref{lem:rates}]
  By Lemma~\ref{lem1} we obtain that for any $t\in[0,T]$,
  \begin{equation}
    \E\Big[ W_\delta\Big( \frac{1}{N}\sum\limits_{i=1}^N \delta_{\underbar{X}_t^i},\mathcal{L}(\underbar{X}_t^1) \Big)^\delta \Big] \lesssim \varepsilon_N,\label{boundW_d}
  \end{equation}
  where $(\varepsilon_N)_{N\in\N}$ is given by \eqref{def:varepsilon} and the right-hand side does not depend on $t$. Plugging \eqref{boundW_d} into \eqref{ineq:show0} and taking the supremum over $[0,T]$ and maximum over $1,\dots,N$ shows the desired convergence rate of the first term in \eqref{eq:rates}. Then, using this and plugging \eqref{boundW_d} into \eqref{boundW1} gives the desired rate for the second term.
\end{proof}

\begin{proof}[Proof of Lemma~\ref{lem:rates2}]
  \textit{Case (i):} Suppose the kernels $K_\mu,K_\sigma$ and initial condition $X_0$ satisfy Assumption~\ref{ass:kernels2}. By Lemma~\ref{lem1} we obtain that
  \begin{equation}
    \E\Big[ W_1\Big( \frac{1}{N}\sum\limits_{i=1}^N \delta_{\underbar{X}_t^i},\mathcal{L}(\underbar{X}_t^1) \Big) \Big] \lesssim N^{-1/2},\label{boundW_d2}
  \end{equation}
  independently from $t\in[0,T]$. Plugging \eqref{boundW_d2} into \eqref{ineq:convRHS} and \eqref{ineq:settingII2} yields the statement.

  \medskip
  \textit{Case (ii):} Suppose the kernels $K_\mu,K_\sigma$ and initial condition $X_0$ satisfy Assumption~\ref{ass:kernels3}. Plugging \eqref{boundW_d2} into the analogues of \eqref{ineq:convRHS} and \eqref{ineq:settingII2} yields the statement.
\end{proof}


\providecommand{\bysame}{\leavevmode\hbox to3em{\hrulefill}\thinspace}
\providecommand{\MR}{\relax\ifhmode\unskip\space\fi MR }
\providecommand{\MRhref}[2]{%
  \href{http://www.ams.org/mathscinet-getitem?mr=#1}{#2}
}
\providecommand{\href}[2]{#2}

\end{document}